\documentclass[11pt]{amsart}

\usepackage{amsmath,amssymb,amsthm}
\usepackage{graphicx}
\usepackage{tikz}
\usepackage{mathrsfs}

\usetikzlibrary{decorations.pathreplacing,decorations.markings}

\tikzset{
  on each segment/.style={
    decorate,
    decoration={
      show path construction,
      moveto code={},
      lineto code={
        \path [#1]
        (\tikzinputsegmentfirst) -- (\tikzinputsegmentlast);
      },
      curveto code={
        \path [#1] (\tikzinputsegmentfirst)
        .. controls
        (\tikzinputsegmentsupporta) and (\tikzinputsegmentsupportb)
        ..
        (\tikzinputsegmentlast);
      },
      closepath code={
        \path [#1]
        (\tikzinputsegmentfirst) -- (\tikzinputsegmentlast);
      },
    },
  },
  mid arrow/.style={postaction={decorate,decoration={
        markings,
        mark=at position .5 with {\arrow[#1]{stealth}}
      }}},
}

\newcommand{\oarc}[4]{
\draw[thick, postaction={on each segment={mid arrow}}] (#1,#2) ..controls (#1 + .2,#2 + .7) and (#3 - .2,#4 + .7) .. (#3,#4);
}

\def\Z{{\mathbb Z}}
\def\C{{\mathbb C}}
\def\R{{\mathbb R}}

\def\A{{\mathcal A}}
\def\cl{\mathcal}

\def\s{{\sigma}}
\def\t{{\tau}}
\def\k{{\kappa}}

\def\a{\alpha}

\def\ar{\operatorname{ar}}
\def\mr{\operatorname{mr}}

\newcommand*{\smallp}[1]{\scalebox{.75}{\ensuremath#1}}
\newcommand*{\subsmallp}[1]{\scalebox{.5}{\ensuremath#1}}

\newcommand{\pp}[2][p]{\imath_{\hspace*{-1pt}#1}\hspace*{-.5pt}\smallp(#2\smallp)}
\newcommand{\subpp}[2][p]{\imath_{\hspace*{-1pt}#1}\hspace*{-.5pt}\subsmallp(#2\subsmallp)}

\newcommand\id{\operatorname{id}}

\newcommand\Aut{\operatorname{Aut}}

\newcommand\diag{\operatorname{diag}}
\newcommand\perm{\operatorname{perm}}

\newcommand{\comment}[1]{}
\newcommand{\al}[1]{\begin{align*}#1\end{align*}}

\newtheorem{thm}{Theorem}[section]
\newtheorem{lem}[thm]{Lemma}
\newtheorem{prop}[thm]{Proposition}
\newtheorem{cor}[thm]{Corollary}
\newtheorem{conj}[thm]{Conjecture}

\theoremstyle{definition}
\newtheorem{defn}[thm]{Definition}
\newtheorem{rem}[thm]{Remark}

\makeatletter
\providecommand\@dotsep{5}
\def\listtodoname{List of Todos}
\def\listoftodos{\@starttoc{tdo}\listtodoname}
\makeatother

\begin{document}

\title{Augmentation Rank of Satellites with Braid Pattern}

\author{Christopher R. Cornwell}
\author{David R. Hemminger}

\begin{abstract}
Given a knot $K$ in $S^3$, a question raised by Cappell and Shaneson asks if the meridional rank of $K$ equals the bridge number of $K$. Using augmentations in knot contact homology we consider the persistence of equality between these two invariants under satellite operations on $K$ with a braid pattern. In particular, we answer the question in the affirmative for a large class of iterated torus knots.
\end{abstract}

\maketitle


\section{Introduction}
Let $K$ be an oriented knot in $S^3$ and denote by $\pi_K$ the fundamental group of its complement $\overline{S^3\setminus n(K)}$, with some basepoint. We call an element of $\pi_K$ a \emph{meridian} if it is represented by the oriented boundary of a disc, embedded in $S^3$, whose interior intersects $K$ positively once. The group $\pi_K$ is generated by meridians; the \emph{meridional rank} of $K$, written $\mr(K)$, is the minimal size of a generating set containing only meridians. 

Choose a height function $h:S^3\to\R$. The \emph{bridge number} of $K$, denoted $b(K)$, is the minimum of the number of local maxima of $h|_{\varphi(S^1)}$ among embeddings $\varphi:S^1\to S^3$ which realize $K$.

By considering Wirtinger's presentation of $\pi_K$ one can show that $\mr(K)\le b(K)$ for any $K\subset S^3$. Whether the bound is equality for all knots is an open question attibuted to Cappell and Shaneson \cite[Prob. 1.11]{Kir95}. Equality is known to hold for some families of knots due to work of various authors (\cite{BZ,Cor13b,RZ}).

Here we study \emph{augmentations} of $K$, which are maps that arise in the study of knot contact homology. To each augmentation is associated a rank and there is a maximal rank of augmentations of a given $K$, called the \emph{augmentation rank} $\ar(K)$. For any $K$ the inequality $\ar(K)\le \mr(K)$ holds (see Section \ref{SecBG_AugRk}). We discuss the behavior of $\ar(K)$ under satellite operations with a braid pattern. 

To be precise, denote the group of braids on $n$ strands by $B_n$ and write $\hat{\beta}$ for the \emph{braid closure} of a braid $\beta$ (see Section \ref{SecBG}, Figure \ref{fig:BClosure}). We write $\imath_{\hspace*{-1pt}n}$ for the identity in $B_n$. 

Throughout the paper we let $\alpha\in B_k$ and $\gamma\in B_p$ and set $K = \hat{\alpha}$. We assume our braid closures are a (connected) knot. Note that $\ar(K)\le k$.

\begin{defn} Let $\pp\alpha$ be the braid in $B_{kp}$ obtained by replacing each strand of $\alpha$ by $p$ parallel copies (in the blackboard framing). Let $\bar\gamma$ be the inclusion of $\gamma$ into $B_{kp}$ by the map $\s_i\mapsto\s_i, 1\le i\le p-1$. Set $\gamma\smallp(\alpha\smallp) = \pp\alpha\bar\gamma$. The \emph{braid satellite} of $K$ associated to $\alpha, \gamma$ is defined as $K(\alpha,\gamma) = \widehat{\gamma\smallp(\alpha\smallp)}$.
\label{defn:BraidSat}
\end{defn}

\begin{figure}[ht]
\begin{tikzpicture}[scale=.7,>=stealth]
    \draw (0-.5,-1) rectangle (3-.5,2);
    \foreach \y in {1.5,0.5,-0.5}
      \draw[->]
          (-0.5-.5,\y) -- (0-.5,\y)
          (3-.5,\y) -- (4-.5,\y);
    \draw (1.5-.5,.5) node {$\alpha$};
    \draw (6,-1) rectangle (9,2);
    \foreach \y in {1.5,0.5,-0.5}
      \foreach \p in {0.15,0.05,-0.05,-0.15}
        \draw
            (5.5,\y+\p) -- (6,\y+\p)
            (9,\y+\p) -- (9.5,\y+\p);
    \draw (7.5,.5) node {$\pp[4]{\alpha}$};
    \draw (9.5,1.25) rectangle (10,1.75);
    \foreach \y in {0.5,-0.5}
      \foreach \p in {0.15,0.05,-0.05,-0.15}
        \draw[->] (9.5,\y+\p) -- (10.25,\y+\p);
      \foreach \p in {0.15,0.05,-0.05,-0.15}
        \draw[->] (10,1.5+\p) -- (10.25,1.5+\p);
    \draw (9.75,1.5) node {\small $\gamma$};
    \draw[dashed,line width=1.5pt,->] 
      (3.25,2.2)  ..controls (3.6,2.8) and (5.05,2.8) ..(5.4,2.2);
  \end{tikzpicture}
  \caption{Constructing $\gamma(\alpha)$ from $\alpha$; case $p=4$.}
  \label{FigBraidSat}
  \end{figure}
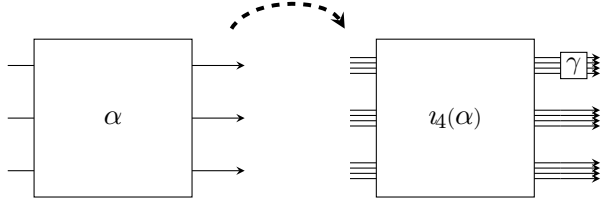

As defined $K(\alpha,\gamma)$ depends on the choice of $\alpha$. However, the construction is more intrinsic if we require the index $k$ of $\alpha$ to be minimal among braid representatives of $K$ (see Section \ref{SecAsSatelliteOp}). 

Note that if $\hat{\alpha}$ and $\hat{\gamma}$ are each a knot, $K(\alpha,\gamma)$ is also. Our principal result is the following.

\begin{thm}\label{main}
If $\alpha\in B_k$ and $\gamma\in B_p$ are such that $\ar(\hat{\alpha})=k$ and $\ar(\hat{\gamma})=p$, then $\ar(K(\alpha,\gamma))=kp$.
\end{thm}

A corollary of Theorem \ref{main} involves Cappell and Shaneson's question for iterated torus knots. Let ${\bf p}=(p_1,\ldots,p_n)$ and ${\bf q}=(q_1,\ldots,q_n)$ be integral vectors with $p_i >0$ for all $1\le i \le n$. We write $T({\bf p},{\bf q})$ for the $({\bf p},{\bf q})$ \emph{iterated torus knot}, defined as follows. 

By convention take $T(\emptyset,\emptyset)$ as the unknot, then define $T({\bf p},{\bf q})$ inductively. Let $\hat{\bf p}, \hat{\bf q}$ be the truncated lists obtained from ${\bf p},{\bf q}$ by removing the last integer in each. If $\alpha$ is a braid of minimal index such that $T(\hat{\bf p},\hat{\bf q}) = \hat{\alpha}$ then define $T({\bf p},{\bf q}) = K(\alpha,(\s_1\ldots\s_{p_n-1})^{q_n})$. 

We remark that $T({\bf p},{\bf q})$ is a cable of $T(\hat{\bf p},\hat{\bf q})$, but not the $(p_n,q_n)$-cable in the traditional Seifert framing.

\begin{cor}\label{cor:iteratedCables}
Given integral vectors ${\bf p}$ and ${\bf q}$, suppose that $|p_i|<|q_i|$ and $\text{gcd}(p_i,q_i)=1$ for each $1\le i\le n$. Then 
      \[\ar(T({\bf p},{\bf q}))=\mr(T({\bf p},{\bf q})) = b(T({\bf p},{\bf q}))=p_1p_2\ldots p_n.\]
\end{cor}

The assumption $|p_i|<|q_i|$ is needed for the hypothesis of Theorem \ref{main}, that the associated braids have closures with augmentation rank equal to the braid index. This requirement is not a deficiency of our techniques; there are cables of $(n,n+1)$ torus knots which do not attain the large augmentation rank in Corollary \ref{cor:iteratedCables}.

\begin{thm}\label{ThmNNPlus1}Given $p>1$ and $n>1$, $\ar(T((n,p),(n+1,1))) < np$.
\end{thm}

It is natural to wonder if the augmentation rank is multiplicative under weaker assumptions on $\alpha, \gamma$ than those in Theorem \ref{main}. The following is a possible generalization.

\begin{conj}Suppose $K=\hat{\alpha}$ for $\alpha\in B_k$, and that $\alpha$ has minimal index among braids with the same closure. Let $\gamma\in B_p$. Then $\ar(K(\alpha,\gamma)) \ge \ar(\hat{\alpha})\ar(\hat{\gamma})$.
\label{ConjSuperMultipl}
\end{conj}

\begin{rem}There are examples when the inequality of Conjecture \ref{ConjSuperMultipl} is strict (see Section \ref{SecComments}).
\label{RemStrictlySuper}
\end{rem}

The paper is organized as follows. Section \ref{SecAsSatelliteOp} relates braid satellites to existing conventions on satellite operators. In Section \ref{SecBG} we give the needed background in knot contact homology, specifically Ng's cord algebra, and discuss augmentation rank and the relationship to meridional rank. Section \ref{SecBG_AugExist} reviews techniques used in the proof of Theorem \ref{main}. Section \ref{SecMain} is devoted to the proof of Theorem \ref{main}, its requisite supporting lemmas, and Corollary \ref{cor:iteratedCables}. Finally, Section \ref{SecComments} considers the sharpness of our results. We prove Theorem \ref{ThmNNPlus1} and briefly discuss the more general case, Conjecture \ref{ConjSuperMultipl}.


\section*{Acknowledgements}
The first author was supported in part by an AMS-Simons travel grant and is very grateful for this program. The second author was supported in part by a grant from the PRUV Fellowship program at Duke University, and thanks David Kraines and the Duke Math Department for organizing the PRUV program. Both authors would like to thank Lenhard Ng for his consultation and helpful comments.


\section{Satellite operators and the braid satellite}
\label{SecAsSatelliteOp}

Definition \ref{defn:BraidSat} of the braid satellite $K(\alpha,\gamma)$ produces a satellite of $\hat{\alpha}$. As defined, the resulting satellite depends on the braid representative of $\hat{\alpha}$. We remark here how to avoid this ambiguity.

A tubular neighborhood of an oriented knot $K$ has a standard identification with $S^1\times D^2$ determined by an oriented Seifert surface that $K$ bounds. Given a knot $P\subset S^1\times D^2$, as per the usual convention, let $P(K)$ be the satellite of $K$ with pattern $P$ obtained with this framing.

\begin{prop}\label{PropAsSatelliteOp}Given a knot $K$ and a braid $\gamma\in B_p$, let $\omega$ be the writhe of some minimal index closed braid representing $K$. Let $P\subset S^1\times D^2$ be the braid closure of $\Delta^{2\omega}\gamma$, where $\Delta^2$ is the full twist in $B_p$. Then $K(\alpha,\gamma) = P(K)$ for any minimal index braid $\alpha$ with $K=\hat{\alpha}$.
\end{prop}
\begin{proof}The principal observation is that, since the Jones conjecture holds \cite{DP12,LM12}, the writhe of $\alpha$ must be $\omega$. Thus the blackboard framing of the closure of $\pp\alpha\bar\Delta^{-2\omega}$ agrees with the $(p,0)$-cable of $K$ (with Seifert framing).
\end{proof}

We note, the satellite $T({\bf p},{\bf q})$ corresponds to the $(p_n,p_n\omega_n+q_n)$-cable of $T(\hat{\bf p},\hat{\bf q})$, where $\omega_n$ is defined inductively by $\omega_n=p_{n-1}\omega_{n-1}+(p_{n-1}-1)q_{n-1}$ and $\omega_1=0$.

Concerning the bridge number of $K(\alpha,\gamma)$, a result of Schubert \cite{Schub} (see \cite{Schul} also) states that if $K$ is not the unknot and $P(K)$ is a satellite such that $P$ has winding number $p$, then $b(P(K)) \ge p\ b(K)$. Since $K(\alpha,\gamma) = \widehat{\gamma\smallp(\alpha\smallp)}$, it has bridge number at most $kp$ and thus $b(K(\alpha,\gamma)) = kp$ whenever $b(\hat\alpha) = k$. From this we see $b(T({\bf p},{\bf q})) = p_1p_2\cdots p_n$, provided $p_1 < q_1$.


\section{Background}
\label{SecBG}

  We review in Section \ref{SecBG_KCHdef} the construction of $HC_0(K)$ from the viewpoint of the combinatorial knot DGA, which was first defined in \cite{Ng08}; our conventions are those given in \cite{Ng12}. In Section \ref{SecBG_AugRk} we discuss augmentations in knot contact homology and their rank, which gives a lower bound on the meridional rank of the knot group. Section \ref{SecBG_AugExist} contains a discussion of techniques from \cite{Cor13a} that we use to calculate the augmentation rank.

  Throughout the paper we orient $n$-braids in $B_n$ from left to right, labeling the strands $1,\ldots, n$, with 1 the topmost and $n$ the bottommost strand. We work with Artin's generators $\{\sigma_i^{\pm}$, $i=1,\ldots,n-1\}$ of $B_n$, where in $\s_i$ only the $i$ and $i+1$ strands interact, and they cross once in the manner depicted in Figure \ref{fig:BraidGens}.
      \begin{figure}[ht]
        \begin{tikzpicture}[scale=0.4,>=stealth]
        \draw[->] (0,2) -- node[at start,left]{$i$} (3,0);
        \draw[draw=white,very thick,double=black,->] (0,0) -- node[at start,left]{$i+1$} (3,2);
        \draw (1.5,-0.5) node[below] {$\sigma_i^{-1}$};

        \draw[->] (6,0) -- (9,2);
        \draw[draw=white,very thick,double=black,->] (6,2) -- (9,0);
        \draw (7.5,-1) node[below] {$\sigma_i$};
        \end{tikzpicture}
        \caption{Generators of $B_n$}
        \label{fig:BraidGens}
      \end{figure}
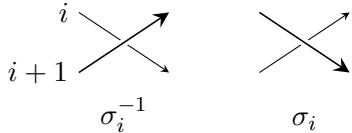
    Given a braid $\beta\in B_n$, the braid closure $\hat{\beta}$ of $\beta$ is the link obtained as shown in Figure \ref{fig:BClosure}. The \emph{writhe} (or algebraic length) of $\beta$, denoted $\omega(\beta)$, is the sum of exponents of the Artin generators in a word representing $\beta$.

    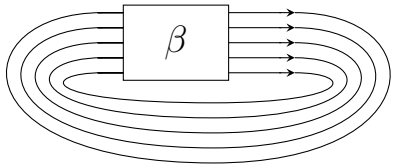
\begin{figure}[ht]
      \begin{tikzpicture}[scale=0.4,>=stealth]
        \draw(0,-.25)--(0,2.25)--(3.5,2.25)--(3.5,-.25)--cycle;
        \draw(1.75,1) node {{\Large $\beta$}};
        \foreach \y in {0,.5,1,1.5,2}
        \foreach \y in {0,.5,1,1.5,2}
          \draw(-.85,\y) -- (0,\y);
        \foreach \y in {0,.5,1,1.5,2} 
          \draw[thin] (3.5,\y) -- (5.25,\y);
        \foreach \y in {0,.5,1,1.5,2} 
          \draw[->] (5.18,\y) -- (5.7,\y);
        \foreach \y in {0,.5,1,1.5,2}
          \draw (5.7,\y) ..controls (7.7+\y*1.3,\y) and (7.7+\y*1.3,-\y-1) ..(3,-\y-1)
                  ..controls (-3-\y*1.3,-\y-1) and (-2.7-\y*1.3,\y)..(-.7,\y);
      \end{tikzpicture}
      \caption{The braid closure of $\beta$}
      \label{fig:BClosure}
    \end{figure}

\subsection{Knot contact homology}
\label{SecBG_KCHdef}

  We review the construction of the combinatorial knot DGA of Ng (in fact, we discuss only the degree zero part as this will suffice for our purposes). This DGA was defined in order to be a calculation of knot contact homology and was shown to be so in \cite{EENS12} (see \cite{Ng12} for more details). Let $\A_n$ be the noncommutative unital algebra over $\Z$ freely generated by $a_{ij}$, $1\le i\ne j\le n$. We define a homomorphism $\phi : B_n \rightarrow\Aut \A_n$ by defining it on the generators of $B_n$:

  \begin{equation}
  \phi_{\s_k}\colon
  \left\{
       \begin{array}{lr}
         a_{ij}\mapsto a_{ij} & i,j\ne k,k+1\\
         a_{k+1,i}\mapsto a_{ki} & i\ne k,k+1\\
         a_{i,k+1}\mapsto a_{ik} & i\ne k,k+1\\
         a_{k,k+1}\mapsto -a_{k+1,k} & \\
         a_{k+1,k}\mapsto -a_{k,k+1} & \\
         a_{ki}\mapsto a_{k+1,i} - a_{k+1,k}a_{ki} & i\ne k,k+1\\
         a_{ik}\mapsto a_{i,k+1} - a_{ik}a_{k,k+1} & i\ne k,k+1\\
       \end{array}
  \right.
  \label{DefnPhiMap}
  \end{equation}

  Let $\iota\colon B_n \rightarrow B_{n+1}$ be the inclusion $\s_i\mapsto\s_i$ so that the $(n+1)$ strand does not interact with those from $\beta\in B_n$, and define $\phi_\beta^*\in \Aut \A_{n+1}$ by $\phi_\beta^* = \phi_{\iota\subsmallp(\beta\subsmallp)}$. We then define the $n\times n$ matrices $\Phi_\beta^L$ and $\Phi_\beta^R$ with entries in $\A_n$ by

  $$\phi_\beta^*(a_{i,n+1}) = \sum_{j=1}^n(\Phi_\beta^L)_{ij}a_{j,n+1}$$

  $$\phi_\beta^*(a_{n+1,i}) = \sum_{j=1}^na_{n+1,j}(\Phi_\beta^R)_{ji}$$

  Finally, let $R_0$ be the Laurent polynomial ring $\Z[\lambda^{\pm1},\mu^{\pm1}]$ and define matrices $\bf{A}$ and $\bf{\Lambda}$ over $R_0$ by

  \begin{equation}
  {\bf A}_{ij} = 
  \left\{
       \begin{array}{lr}
        a_{ij} & i<j\\
        -\mu a_{ij} & i>j\\
        1-\mu & i = j\\
       \end{array}
  \right.
  \label{def:Amatrix}
  \end{equation}
  \begin{equation}
  {\bf \Lambda} = \diag[\lambda\mu^{\omega(\beta)},1,\ldots,1].
  \label{defn:Lambda}
  \end{equation}

  \begin{defn}
  Suppose that $K$ is the closure of $\beta\in B_n$. Define $\mathcal{I}\subset \A_n\otimes R_0$ to be the ideal generated by the entries of $\bf{A} - \Lambda\cdot\Phi_\beta^L\cdot \bf{A}$ and $\bf{A} - \bf{A}\cdot\Phi_\beta^R\cdot\Lambda^{-1}$. The \emph{degree zero homology of the combinatorial knot DGA} is $\operatorname{HC}_0(K) = (\A_n\otimes R_0)/\mathcal{I}$.
  \label{defn:HC_0}
  \end{defn}
  
\subsection{Spanning arcs}
\label{SecSpanArcs}

  The proofs in Sections \ref{SecMain} and \ref{SecComments} require a number of computations of $\phi_\beta$ (and of $\phi_{\beta}^\ast$, for computing $\Phi_\beta^L$) for particular braids. Such computations are benefited by an alternate description of the automorphism, which we now explain.

  \begin{defn}Given $n>0$, let $D_n$ be a disk in $\C$ containing points $P=\{1,2,\ldots,n\}$ on the real line. A \emph{spanning arc} of $D_n$ is the isotopy class relative to $P$ of an oriented embedded path in $D$ which begins and ends in $P$. We define $\mathscr S_n$ as the associative ring freely generated by spanning arcs of $D_n$ modulo the ideal generated by the relation in Figure \ref{FigRelnPathAlg}. Denote by $c_{ij}\in\mathscr S_n$ the element represented by a spanning arc contained in the upper half-disk beginning at $i$ and ending at $j$.
  \label{DefnSpanningArcAlg}
  \end{defn}
  We understand the spanning arcs in Figure \ref{FigRelnPathAlg} to agree outside of a neighborhood of the depicted point in $P$.

  \begin{figure}[ht]
\begin{tikzpicture}[scale=0.5,>=stealth]
  \filldraw (-1,1) circle (3pt)
        (3.5,1) circle (3pt)
        (7.5,1) circle (3pt)
        (8.75,1)circle(3pt);
  \filldraw   (8.125,1) circle(1pt);
  \draw[thick,->] 
        (-2,1) ..controls (-1.5,1) and (-1.5,.5) ..(-1,.5)
             ..controls (-.5,.5) and (-.5,1) ..(0,1);
  \draw[thick,->] 
        (2.5,1) ..controls (3,1) and (3,1.5) ..(3.5,1.5)
             ..controls (4,1.5) and (4,1) ..(4.5,1);
  \draw[thick,->] (6.5,1) -- (7.5,1);
  \draw[thick,->] (8.75,1) -- (9.75,1);
  \draw (1.25,1) node {$=$}
      (5.5,1) node {$-$};
  \foreach \x in {-2,2.5,6.5,8.75}
  \draw (\x,1) node[left] {$\big[$};
  \foreach \x in {0,4.5,7.5,9.75}
  \draw   (\x,1) node[right] {$\big]$};     
\end{tikzpicture}
\caption{Relation in $\mathscr S_n$}
\label{FigRelnPathAlg}
  \end{figure}

We consider $\beta$ as a mapping class of $(D,P)$ and denote by $\beta\cdot c$ the image of the spanning arc $c$. By convention $\s_k$ acts by rotating $k$ and $k+1$ about their midpoint in counter-clockwise fashion. It was shown in \cite[Section 2]{Ng05b} that there is a unique, well-defined map $\chi$ which sends each spanning arc of $D_n$ to an element of $\A_n$ such that 
    \begin{itemize}
      \item[(i)] $\chi(\beta\cdot c) = \phi_{\beta}(\chi(c))$ for any spanning arc $c$ and $\beta\in B_n$;
      \item[(ii)] $\chi(c_{ij})=a_{ij}$ if $i<j$, $\chi(c_{ij})=-a_{ij}$ if $i>j$.
    \end{itemize}
Furthermore, $\chi$ factors through $\mathscr S_n$, is injective, and by the relation in Figure \ref{FigRelnPathAlg} the value of $\phi_{\beta}(a_{ij})$ can be determined from (i) and (ii). This constitutes an essential technique for our calculations of $\phi_\beta$.

    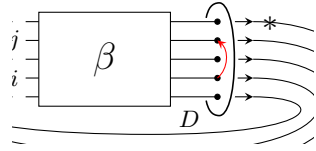
\begin{figure}[ht]
    \begin{tikzpicture}[scale=0.5,>=stealth]
      \clip (-.7,-1.25) rectangle (7.3,2.5);
      \draw(0,-.25)--(0,2.25)--(3.5,2.25)--(3.5,-.25)--cycle;
      \draw(1.75,1) node {{\Large $\beta$}};
      \foreach \y in {0,.5,1,1.5,2}
        \filldraw (4.75,\y) circle (2pt);
      \foreach \y in {0,.5,1,1.5,2}
        \draw(-.35,\y) -- (0,\y);
      \foreach \y in {0,.5,1,1.5,2} 
        \draw[thin] (3.5,\y) -- (4.75,\y);
      \foreach \y in {0,.5,1,1.5,2} 
        \draw[->] (5.18,\y) -- (5.7,\y);
      \foreach \y in {0,.5,1,1.5,2}
        \draw (5.7,\y) ..controls (7.7+\y*1.3,\y) and (7.7+\y*1.3,-\y-1) ..(3,-\y-1)
                ..controls (-3-\y*1.3,-\y-1) and (-2.7-\y*1.3,\y)..(-.7,\y);
      \draw[red,->]
          (4.75,.5) ..controls (5.05,.75) and (5.05,1.25) ..(4.75,1.5);
      \draw[draw=white,double=black,very thick]
          (4.45,2.1)  ..controls (4.5,2.35) and (4.6,2.5) .. (4.8,2.5)
                ..controls (5,2.5) and (5.2,2) .. (5.2 ,1)
                ..controls (5.2,0) and (5,-.5) .. (4.8,-.5)
                ..controls (4.6,-.5) and (4.5,-.35)..(4.45,-.1);
      \draw   (4,-.6) node {{\footnotesize $D$}};
      \draw (6.15,1.95) node {{\large $\ast$}};
      \draw (-.65,.5) node {{\footnotesize $i$}}
          (-.65,1.5) node{{\footnotesize $j$}};
    \end{tikzpicture}
    \caption{Cord $c_{ij}$ of $K=\hat \beta$}
    \label{FigA_nGens}
    \end{figure} 
  Computations of $\Phi_\beta^L$ are carried out in likewise manner, including $\beta$ into $B_{n+1}$ and considering spanning arcs $c_{j,n+1}$, $1\le j\le n$ of $D_{n+1}$. We will distinguish this situation by relabeling $n+1$ (and corresponding indices) with the symbol $\ast$. In figures, we put the point $\ast$ at the boundary of $D$.

  It will be convenient for us in Section \ref{SecMain} to consider the free left $\A_n$-module $\A_n^L = \A_n\langle a_{1\ast},\ldots,a_{n\ast} \rangle$ and right $\A_n$-module $\A_n^R = \langle a_{\ast 1},\ldots,a_{\ast n} \rangle\A_n$, which are each contained in $\A_{n+1}$. By definition, $\Phi_\beta^L$ (respectively $\Phi_\beta^R$) is the matrix in the above basis for the $\A_n$-automorphism of $\A_n^L$ (respectively $\A_n^R$) determined by the image of the basis under $\phi^\ast_\beta$ (which differs from the non-linear map given by restricting $\phi^\ast_\beta$ to these submodules).

  Finally, as we are considering braid satellites $K(\alpha,\gamma)$ with $\gamma\in B_p$ our perspective often considers the points in $D_{kp}$ as $k$ groups of $p$ points each. We find it convenient in figures of spanning arcs in $\mathscr S_{kp}$ to reflect this point of view. To do so, for each $i=0,\ldots,k-1$, we depict the points $\{ip+1,\ldots, (i+1)p\}$ by a horizontal segment, and if a spanning arc ends at $ip+s$ for $1\le s\le p$, it is depicted ending on the $(i+1)^{st}$ segment with a label $s$ (see example in Figure \ref{FigExSpanArckp}).

  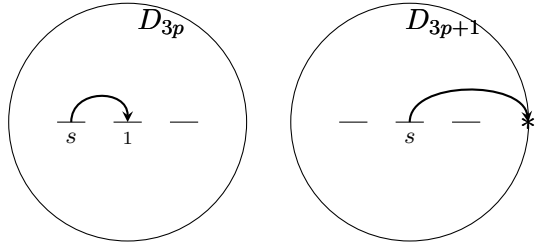
\begin{figure}[ht]
      \begin{tikzpicture}[scale=.75,>=stealth]
        \foreach \c in {0,5} \draw (\c,1) circle (60pt);
        \foreach \c in {0,5} \foreach \p in {-1.25,-.25,0.75} \draw (\c+\p,1) -- (\c+\p+.5,1);
        \foreach \c in {0,5}  
        \draw   (7.1,1) node {{\large $\ast$}}
                (0.6,2.75) node {$D_{3p}$}
                (5.6,2.75) node {$D_{3p+1}$};
        \draw[->,thick] (-1,1)  ..controls (-1,1.6) and (0,1.6) ..(0,1);
        \draw[->,thick] (5,1)  ..controls (5,1.75) and (7,1.75) ..(7.1,1);
        \foreach \c in {-1,5} \draw (\c,1) node[below] {{\footnotesize $s$}};
        \draw (0,1) node[below] {{\tiny 1}};
      \end{tikzpicture}
      \caption{Spanning arcs $c_{s,p+1}$ and $c_{p+s,\ast}$, $1\le s\le p$.}
      \label{FigExSpanArckp}
    \end{figure}

  Let $\text{perm}:B_n\to S_n$ denote the homomorphism from $B_n$ to the symmetric group sending $\s_k$ to the simple transposition interchanging $k, k+1$.

  \begin{lem} For some $\beta\in B_n$ and $1\le i\ne j\le n$, consider $(\Phi_\beta^L)_{ij}\in \A_n$ as a polynomial expression in the (non-commuting) variables $\{a_{kl}, 1\le k\ne l\le n\}$. Writing $i_0=\text{perm}(\beta)(i)$, every monomial in $(\Phi_\beta^L)_{ij}$ is a constant times $a_{i_0i_1}a_{i_1i_2}\ldots a_{i_{l-1},j}$ for some $l\ge 0$, the monomial being a constant if $l=0$ and only if $i_0=j$.
  \label{lem:monomial}
  \end{lem}
  \begin{proof}We consider the spanning arc $\beta\cdot c_{i,\ast}$ which begins at $i_0$ and ends at $\ast$. Applying the relation in Figure \ref{FigRelnPathAlg} to the path equates it with a sum (or difference) of another path with the same endpoints and a product of two paths, the first beginning at $i_0$ and the other ending at $\ast$. A finite number of applications of this relation allows one to express the path as a polynomial in the $c_{kl}, 1\le k\ne l\le n$ where each monomial has the form $c_{i_0i_1}\ldots c_{i_{l-1},j}c_{j,\ast}$ for some $j$. The result follows from $\phi^\ast_\beta(a_{i,\ast}) = \phi^\ast_\beta(\chi(c_{i,\ast})) = \chi(\beta\cdot c_{i,\ast})$.
  \end{proof}

\subsection{Augmentations and augmentation rank}
\label{SecBG_AugRk}

  Augmentations of a differential graded algebra $(\A,\partial)$ are graded maps $(\A,\partial)\to (\C,0)$ that intertwine the differential (here $\C$ has grading zero). For our setting, if $\beta\in B_n$ is a braid representative of $K$, such a map corresponds precisely to a homomorphism $\epsilon:\A_n\otimes R_0\to\C$ such that $\epsilon$ sends elements of $\mathcal I$ to zero (see Definition \ref{defn:HC_0}).

  \begin{defn}
  Suppose that $K$ is the closure of $\beta\in B_n$. An \emph{augmentation} of $K$ is a homomorphism $\epsilon: \A_n\otimes R_0\rightarrow \C$ such that each element of $\mathcal I$ is sent by $\epsilon$ to zero.
  \label{defn:Aug}
  \end{defn}

  A correspondence between augmentations and certain representations of the knot group $\pi_K$ were studied in \cite{Cor13a}. Recall that $\pi_K$ is generated by meridians, which for a knot are all conjugate. Fix some meridian $m$.

  \begin{defn}
  For any integer $r\ge1$, a homomorphism $\rho:\pi_K\to\text{GL}_r\C$ is a \emph{KCH representation} if $\rho(m)$ is diagonalizable and has an eigenvalue of 1 with multiplicity $r-1$. We call $\rho$ a \emph{KCH irrep} if it is irreducible.
  \label{defn:KCHReps}
  \end{defn}

  In \cite{Ng08}, Ng describes an isomorphism between $HC_0(K)$ and an algebra constructed from elements of $\pi_K$. As discussed in \cite{Ng12} a KCH representation $\rho:\pi_K\to\text{GL}_r\C$ induces an augmentation $\epsilon_\rho$ of $K$. Given an augementation, the first author showed how to construct a KCH representation that induces it. In fact, we have the following rephrasing of results from \cite{Cor13a}.

  \begin{thm}[\cite{Cor13a}]
  Let $\epsilon:\A_n\otimes R_0\to\C$ be an augmentation with $\epsilon(\mu)\ne 1$. There is a KCH irrep $\rho:\pi_K\to\text{GL}_r\C$ such that $\epsilon_\rho=\epsilon$. Furthermore, for any KCH irrep $\rho:\pi_K\to\text{GL}_r\C$ such that $\epsilon_\rho = \epsilon$, the rank of $\epsilon({\bf A})$ equals $r$.
  \label{thm:AugKCH_Corresp}
  \end{thm}

  Considering Theorem \ref{thm:AugKCH_Corresp} we make the following definition.

  \begin{defn}
  The \emph{rank} of an augmentation $\epsilon:\A_n\otimes R_0\to\C$ with $\epsilon(\mu)\ne 1$ is the rank of $\epsilon({\bf A})$. Given a knot $K$, the \emph{augmentation rank of} $K$, denoted $\text{ar}(K)$, is the maximum rank among augmentations of $K$.
  \label{defn:AugRk}
  \end{defn}

  \begin{rem} By Theorem \ref{thm:AugKCH_Corresp} the set of ranks of augmentations of a given $K$ does not depend on choice of braid representative.
  \end{rem}

  It is the case that $\text{ar}(K)$ is well-defined. That is, given $K$ there is a bound on the maximal rank of an augmentation of $K$.

  \begin{thm}[\cite{Cor13b}] Given a knot $K\subset S^3$, if $g_1,\ldots,g_d$ are meridians that generate $\pi_K$ and $\rho:\pi_K\to\text{GL}_r\C$ is a KCH irrep then $r\le d$.
  \label{thm:DimBound}
  \end{thm}

  As in the introduction, if we denote the meridional rank of $\pi_K$ by $\text{mr}(K)$, then Theorem \ref{thm:DimBound} implies that $\text{ar}(K)\le\text{mr}(K)$. In addition, the geometric quantity $b(K)$ called the bridge index of $K$ is never less than $\text{mr}(K)$. Thus we have the following corollary.
    
\begin{cor}[\cite{Cor13b}] Given a knot $K\subset S^3$,
$$\text{ar}(K)\le\text{mr}(K)\le b(K)$$
\label{cor:DimBound}
\end{cor}

  Hence to verify that $\text{mr}(K)=b(K)$ it suffices to find a rank $b(K)$ augmentation of $K$. Herein we concern ourselves with a setting where $\text{ar}(K)=n$ and there is a braid $\beta\in B_n$ which closes to $K$. This is a special situation, since $b(K)$ is strictly less than the braid index for many knots.

\subsection{Finding augmentations}
\label{SecBG_AugExist}
  The following theorem concerns the behavior of the matrices $\Phi_\beta^L$ and $\Phi_\beta^R$ under the product in $B_n$. It is an essential tool for studying $HC_0(K)$ and is central to our arguments.

  \begin{thm}[\cite{Ng05}, Chain Rule] Let $\beta_1,\beta_2$ be braids in $B_n$. Then $\Phi_{\beta_1\beta_2}^L = \phi_{\beta_1}(\Phi_{\beta_2}^L)\cdot\Phi_{\beta_1}^L$ and $\Phi_{\beta_1\beta_2}^R = \Phi_{\beta_1}^R\cdot\phi_{\beta_1}(\Phi_{\beta_2}^R)$.
  \label{thm:ChainRule}
  \end{thm}

  Another property of $\Phi_\beta^L$ and $\Phi_\beta^R$ that is important to us is the following symmetry. Define an involution $x\mapsto\overline x$ on $\A_n$ (termed \emph{conjugation}) as follows: first set $\overline{a_{ij}}=a_{ji}$; then, for any $x,y\in\A_n$, define $\overline{xy}=\overline y\hspace*{1pt}\overline x$ and extend the operation linearly to $\A_n$.

  \begin{thm}[\cite{Ng05}, Prop.\hspace*{-0.7pt} 6.2]For a matrix of elements in $\A_n$, let $\overline{M}$ be the matrix such that $\left(\overline M\right)_{ij} = \overline{M_{ij}}$. Then for $\beta\in B_n$, $\Phi_\beta^R$ is the transpose of $\overline{\Phi_\beta^L}$.
  \label{thm:Transpose}
  \end{thm}

  The main result of the paper concerns augmentations with rank equal to the braid index of $K$. Define the diagonal matrix $\Delta(\beta)=\text{diag}[(-1)^{w(\beta)},1,\ldots,1]$. From Section 5 of \cite{Cor13a} we have the following.

  \begin{thm}[\cite{Cor13a}] If $K$ is the closure of $\beta\in B_n$ and has a rank $n$ augmentation $\epsilon:\A_n\otimes R_0\to\C$, then 
    \begin{equation}
    \epsilon(\Phi_\beta^L)=\Delta(\beta)=\epsilon(\Phi_\beta^R).
    \label{eqn:FindingAugs}
    \end{equation}
    Furthermore, any homomorphism $\epsilon:\A_n\to\C$ which satisfies (\ref{eqn:FindingAugs}) can be extended to $\A_n\otimes R_0$ to produce a rank $n$ augmentation of $K$. 
  \label{thm:RanknAugs}
  \end{thm}


\section{Main Result}
\label{SecMain}
The proof of Theorem \ref{main} relies heavily on the characterization presented in Theorem \ref{thm:RanknAugs}. We define a homomorphism $\psi:\cl A_{kp}\to\cl A_k\otimes\cl A_p$ which, for $\alpha\in B_k$, suitably simplifies $\Phi_{\subpp\alpha}^L$ and $\Phi_{\subpp\alpha}^R$ when applied to the entries. Given $\gamma\in B_p$, Theorem \ref{thm:ChainRule} then allows us to construct a map that satisfies (\ref{eqn:FindingAugs}) for $\beta=\gamma\smallp(\alpha\smallp)$. The map in question is ``close to'' the tensor product of an augmentation of $\hat\alpha$ and an augmentation of $\hat\gamma$, composed with $\psi$.

Section \ref{MainProof} begins with an intermediate result, Proposition \ref{psiofbp}, followed by the proofs of Theorem \ref{main} and Corollary \ref{cor:iteratedCables}.  In Section \ref{PropAndLemmas} we prove Lemmas \ref{commutes} and \ref{Sigma_n}, which are needed to prove Proposition \ref{psiofbp}.

\subsection{Proof of main result}
\label{MainProof}
We recall the statement of Theorem \ref{main}.

\newtheorem*{main}{Theorem \ref{main}}
\begin{main}
If $\alpha\in B_k$ and $\gamma\in B_p$ are such that $\ar(\hat{\alpha})=k$ and $\ar(\hat{\gamma})=p$, then $\ar(K(\alpha,\gamma))=kp$.
\end{main}

For $1\le i\le kp$, write $i = (q_i-1)p + r_i$, where $1\le r_i \le p$ and $1\le q_i\le k$.  For each generator $a_{ij}\in\A_{kp}, 1\le i\ne j\le kp$, define

\begin{equation}
\psi(a_{ij}) =
  \begin{cases}
         1\otimes a_{r_ir_j} & \colon q_i = q_j\\
         a_{q_iq_j}\otimes 1 & \colon r_i = r_j\\
         0 & \colon (q_i-q_j)(r_i-r_j)<0\\
         a_{q_iq_j}\otimes a_{r_ir_j} & \colon (q_i-q_j)(r_i-r_j)>0\\
  \end{cases},
  \label{defn:psi}
\end{equation}

\noindent which determines an algebra map $\psi\colon \A_{kp} \rightarrow \A_k\otimes \A_p$. Extend $\psi$ to a map $\psi^\ast:\A_{kp}^L\to\A_k^L\otimes\A_p^L$ that takes one canonical basis to another: $\psi^\ast(a_{i\ast}) = a_{q_i,\ast}\otimes a_{r_i,\ast}$ for any $1\le i\le kp$. Note, if we extend conjugation to $\A_k\otimes\A_p$ by applying it to each factor, then $\psi(\overline{a_{ij}}) = \overline{\psi(a_{ij})}$.

\begin{prop}\label{psiofbp}
$\psi\left(\Phi_{\subpp\alpha}^L\right) = \Phi_\alpha^L\otimes I_p$ and $\psi\left(\Phi_{\subpp\alpha}^R\right) = \Phi_\alpha^R\otimes I_p$\\
for any braid $\alpha$.
\end{prop}

A comment on notation is in order. The tensor product (over $\Z$) of $\A_k^L$ and $\A_p^L$ is a left $(\A_k\otimes\A_p)$-module with canonical basis $\{a_{i\ast}\otimes a_{j\ast}\}$. By $\Phi_\alpha^L\otimes I_p$ we mean the matrix in this basis for the $(\A_k\otimes\A_p)$-linear map equal to the tensor product of the map corresponding to $\Phi_\alpha^L$ with the identity on $\A_p^L$. Similarly for $\A_k^R$ and $\A_p^R$.

Proposition \ref{psiofbp} hinges on the following lemma, proved in Section \ref{PropAndLemmas}.

\begin{lem}\label{commutes} For $\alpha\in B_k$ the following diagram commutes.

  \[\begin{tikzpicture}[scale=.75,>=stealth]
    \foreach \m in {0,3.5}
      \draw (\m,2) node {\footnotesize $\A_{kp}^L$};
    \draw (.65,0) node[left] {\footnotesize $\A_k^L\otimes\A_p^L$}
          (2.85,0) node[right] {\footnotesize $\A_k^L\otimes\A_p^L$};
    \draw[->] (.7,2) -- node[above] {\footnotesize $\phi_{\subpp\alpha}^\ast$} (2.8,2);
    \draw[->] (.7,0) -- node[below] {\footnotesize $\phi_\alpha^\ast\otimes\id$} (2.8,0);
    \foreach \p in {0,3.5}
      \draw[->] (\p,1.5) -- node[left] {\footnotesize $\psi^\ast$} (\p,.5);
  \end{tikzpicture}\]

\noindent In particular, $\psi^\ast(\phi_{\subpp\alpha}(a_{i,\ast})) = (\phi_{\alpha} \otimes \id)(\psi^\ast(a_{i,\ast}))$ for any $1\le i\le kp$.
\end{lem}

\begin{proof}[Proof of Proposition \ref{psiofbp}]
The proposition readily follows from Lemma \ref{commutes}. Fixing $\alpha\in B_{k}$ and $1\le i\le kp$, we have

\begin{align*}
\left(\sum_{l=1}^k \left(\Phi_\alpha^L\right)_{q_il}a_{l*}\right)\otimes a_{r_i*} &= \left(\phi^\ast_{\alpha}\otimes \id\right)\psi^\ast\left(a_{i*}\right)\\
&= \psi^\ast\left(\phi^\ast_{\subpp{\alpha}}\left(a_{i*}\right)\right)\\
&= \sum_{j=1}^{kp} \psi\left(\left(\Phi_{\subpp{\alpha}}^L\right)_{ij}\right)\left(a_{q_j*}\otimes a_{r_j*}\right).\\
\end{align*}
Hence $\psi((\Phi_{\subpp{\alpha}}^L)_{ij}) = 0$ if $r_i \ne r_j$ and $\psi((\Phi_{\subpp{\alpha}}^L)_{ij}) = (\Phi_\alpha^L)_{q_iq_j}\otimes 1$ if $r_i = r_j$, since for each $1\le l\le k$ exactly one $j$ satisfies both $r_j=r_i$ and $q_j=l$. We conclude $\psi(\Phi_{\subpp{\alpha}}^L) = \Phi_\alpha^L\otimes I_p$. Since $\Phi_\alpha^R = \overline{\Phi_\alpha^L}^t$ and $\psi(\overline{a_{ij}}) = \overline{\psi(a_{ij})}$, we have $\psi(\Phi_{\subpp{\alpha}}^R) = \Phi_\alpha^R\otimes I_p$ as well.
\end{proof}

\begin{proof}[Proof of Theorem \ref{main}]
By Theorem \ref{thm:RanknAugs} there exist augmentations $\epsilon_k\colon \A_k\otimes R_0 \rightarrow \C$ and $\epsilon_p\colon \A_p\otimes R_0 \rightarrow \C$, for the closures of $\alpha,\gamma$ respectively, such that $\epsilon_k\left(\Phi_\alpha^L\right) = \epsilon_k\left(\Phi_\alpha^R\right) = \Delta(\alpha)$ and $\epsilon_p\left(\Phi_{\gamma}^L\right) = \epsilon_p\left(\Phi_{\gamma}^R\right) = \Delta(\gamma)$. Theorem \ref{thm:RanknAugs} also implies that it suffices to prove that there exists an augmentation $\epsilon\colon \A_{kp}\otimes R_0\rightarrow \C$ such that $\epsilon\left(\Phi_{\gamma\subsmallp(\alpha\subsmallp)}^L\right) = \epsilon\left(\Phi_{\gamma\subsmallp(\alpha\subsmallp)}^R\right) = \Delta(\gamma\smallp(\alpha\smallp))$.

Below we will define a homomorphism $\delta\colon\A_p\rightarrow \C$ such that for each generator $a_{ij}$ we have $\delta(a_{ij}) = \pm \epsilon_p(a_{ij})$, the sign depending on the parity of $w(\alpha)$ and $p$. Let $\pi\colon \C\otimes \C \rightarrow \C$ be the multiplication $a\otimes b\mapsto ab$. Our desired map is defined by $\epsilon = \pi\circ(\epsilon_k\otimes\delta)\circ\psi$.

The Chain Rule theorem gives that
\begin{equation}
\pi\circ(\epsilon_k\otimes\delta)\circ\psi\left(\Phi_{\gamma\subsmallp(\alpha\subsmallp)}^L\right) = \pi\circ(\epsilon_k\otimes\delta)\psi\left(\phi_{\subpp\alpha}\left(\Phi_{\bar\gamma}^L\right)\right)\psi\left(\Phi_{\subpp\alpha}^L\right)
\label{eqn1MainPf}
\end{equation}

Consider how the homomorphism $\phi_{\subpp\alpha}$ acts on spanning arcs. For $1\le i\ne j\le p$, since the points $\{1,\ldots,p\}\in D_{kp}$ are moved as one block by the action of $\pp\alpha$, there is an $0\le m<k$ so that $\phi_{\subpp\alpha}(a_{ij})=a_{i+mp,j+mp}$. As $\psi(a_{i + mp, j+mp})=1\otimes a_{ij}$,

$$\psi\left(\phi_{\subpp\alpha}\left(\Phi_{\bar\gamma}^L\right)\right) = \left(1\otimes \left(\Phi_{\bar\gamma}^L\right)_{ij}\right).$$

Note that while the entries of $\Phi_{\bar\gamma}^L$ are elements of $\A_{kp}$, all of them lie in the image of the natural inclusion of $\A_p$ into $\A_{kp}$, so we regard the entries of the matrix on the right hand side as elements of $\A_k\otimes \A_p$. Returning to the right hand side of (\ref{eqn1MainPf}), by Proposition \ref{psiofbp} we have

\begin{align*}
\pi\circ(\epsilon_k\otimes\delta)\left(\psi\left(\phi_{\subpp\alpha}\left(\Phi_{\bar\gamma}^L\right)\right)\psi\left(\Phi_{\subpp\alpha}^L\right)\right)
    & = \pi\circ(\epsilon_k\otimes\delta)\left(\left(1\otimes \left(\Phi_{\bar\gamma}^L\right)_{ij}\right)\left(\Phi_\alpha^L\otimes I_p\right)\right)\\
    & = \delta\left(\Phi_{\bar\gamma}^L\right)\pi\left(\Delta(\alpha)\otimes I_p\right).
\end{align*}

We are done if we define $\delta$ so that $\delta\left(\Phi_{\bar\gamma}^L\right)\pi\left(\Delta(\alpha)\otimes I_p\right) = \Delta(\gamma\smallp(\alpha\smallp))$.  When $w(\alpha)$ is even $w(\pp\alpha)$ is also, and further $\Delta(\alpha)=I_k$. Letting $\delta = \epsilon_p$ makes 
\begin{equation*}
\delta\left(\Phi_{\bar\gamma}^L\right)\pi\left(\Delta(\alpha)\otimes I_p\right) = \epsilon_p\left(\Phi_{\bar\gamma}^L\right) = \Delta(\bar\gamma) = \Delta(\gamma\smallp(\alpha\smallp)).
\end{equation*}

Suppose $w(\alpha)$ is odd. Define $g\colon \{1,\ldots,p\}\rightarrow \{\pm 1\}$ as follows. Let $x_1 = 1$, and $x_l = \perm(\bar\gamma)(x_{l-1})$ for $1<l\le p$. Since the first $p$ strands of $\bar\gamma$ close to a knot, $\perm(\bar\gamma)$ is given by the $p$-cycle $(x_1x_2\ldots x_p)$. If $p$ is even, we let $g(x_1) = 1$, and $g(x_l) = -g(x_{l-1})$ for $1<l\le p$. If $p$ is odd, let $g(x_1) = g(x_2) = 1$ and $g(x_l) = -g(x_{l-1})$ for $2<l\le p$.

 Define $\delta: \A_p\to \C$ by setting $\delta(a_{ij}) = g(i)g(j)\epsilon_p(a_{ij})$ for $1\le i\ne j\le p$. Fix $i,j$ and consider a monomial $M$ of $\left(\Phi_{\bar\gamma}^L\right)_{ij}$, which is constant if $i>p$ or $j>p$. For $i,j\le p$, writing $i_0=\text{perm}(\bar\gamma)(i)$, Proposition \ref{lem:monomial} implies $M=c_{ij}a_{i_0,j_1}a_{j_1,j_2}\ldots a_{j_m,j}$ for some $j_1,\ldots j_m\in \{1,\ldots,p\}$, possibly being constant if $i_0=j$, implying that 

$$\delta(M) = g(i_0)g(j)\left(\prod_{k=1}^m g(j_k)^2\right)\epsilon_p(M) = g(i_0)g(j)\epsilon_p(M).$$

\noindent For $M$ a constant, $\delta(M) = M = g(i_0)g(j)\epsilon_p(M)$ since $i_0=j$. This holds for each monomial, thus

$$\delta\left(\left(\Phi_{\bar\gamma}^L\right)_{ij}\right) = g(i_0)g(j)\epsilon_p\left(\left(\Phi_{\bar\gamma}^L\right)_{ij}\right).$$

When $p$ is even, $w(\pp\alpha)$ is also even and so the opposite parity of $w(\alpha)$. Our definition of $g$ gives $\delta\left(\left(\Phi_{\bar\gamma}^L\right)_{ii}\right) = -\epsilon\left(\left(\Phi_{\bar\gamma}^L\right)_{ii}\right)$ for $i\le p$. Thus 

$$\delta\left(\Phi_{\bar\gamma}^L\right) = 
\left( \begin{array}{ccc}
(-1)^{w(\bar\gamma)+1} & 0 & 0 \\
0 & -I_{p-1} & 0 \\
0 & 0 & I_{(k-1)p} \end{array} \right)
$$
\noindent and therefore
$$
\delta\left(\Phi_{\bar\gamma}^L\right)\left(\Delta(\alpha)\otimes I_p\right) = \diag[(-1)^{w(\alpha) + w(\bar\gamma) + 1},1\ldots 1] = \Delta(\gamma\smallp(\alpha\smallp))
$$
\noindent as desired.

When $p$ is odd, $w(\pp\alpha)$ is odd and therefore the same parity of $w(\alpha)$. Our definition of $g$ gives that $\delta\left(\left(\Phi_{\bar\gamma}^L\right)_{11}\right) = \epsilon\left(\left(\Phi_{\bar\gamma}^L\right)_{11}\right)$ and $\delta\left(\left(\Phi_{\bar\gamma}^L\right)_{ii}\right) = -\epsilon\left(\left(\Phi_{\bar\gamma}^L\right)_{ii}\right)$ for $1<i\le p$, so 

$$\delta\left(\Phi_{\bar\gamma}^L\right) = 
\left( \begin{array}{ccc}
(-1)^{w(\bar\gamma)} & 0 & 0 \\
0 & -I_{p-1} & 0 \\
0 & 0 & I_{(k-1)p} \end{array} \right)
$$
\noindent and therefore
$$
\delta\left(\Phi_{\bar\gamma}^L\right)\left(\Delta(\alpha)\otimes I_p\right)= \diag[(-1)^{w(\alpha) + w(\bar\gamma)},1\ldots 1] = \Delta(\gamma\smallp(\alpha\smallp))
$$
\noindent as desired. 

There is little difference in the proof that $\epsilon(\Phi_{\gamma\subsmallp(\alpha\subsmallp)}^R) = \Delta(\gamma\smallp(\alpha\smallp))$, except that monomials in $(\Phi_{\bar\gamma}^R)_{ij}$ are of the form $c_{ij}a_{i,j_1}a_{j_1,j_2}\cdots a_{j_k,j'}$ where $j'=\text{perm}(\bar\gamma)(j)$. Applying Theorem \ref{thm:RanknAugs} now completes the proof.
\end{proof}

\begin{proof}[Proof of Corollary \ref{cor:iteratedCables}]
We prove the corollary by induction on the dimensions of the vectors ${\bf p}$ and ${\bf q}$.  If ${\bf p}$ and ${\bf q}$ have one entry, then $T({\bf p},{\bf q})$ is simply the $(p_1,q_1)$-torus knot, and by Theorem 1.3 from \cite{Cor13b} we have $\ar(T({\bf p},{\bf q})) = p_1$.

Suppose that ${\bf p}$ and ${\bf q}$ have $n$ entries and $\ar(T({\bf \hat p},{\bf \hat q})) = p_1p_2\cdots p_{n-1}$.  Choose a braid $\alpha\in B_{p_1p_2\cdots p_{n-1}}$ such that $\hat\alpha = T({\bf \hat p},{\bf \hat q})$, and let $\gamma = (\sigma_1\ldots \sigma_{p_n-1})^{q_n}$.  Theorem 1.3 from \cite{Cor13b} implies that $\ar(\gamma) = p_n$, and since $T({\bf \hat p},{\bf \hat q}) = K(\alpha,\gamma)$, Theorem \ref{main} gives the desired result.
\end{proof}

\subsection{Supporting Lemmas}
\label{PropAndLemmas} 


In this section we prove Lemma \ref{commutes} for which we make some definitions.  Set $X_{m,l} = \{m,m+1,\ldots,m+l-1\}$ for any $m,l>0$. For a given $Y\subseteq X_{m,l}$ we denote elements of $Y$ by $\{y_1,\ldots,y_k\}$ so that $y_1<\ldots <y_k$.  Suppose $1\le i\ne j\le kp+1$. If $i,j\not\in X_{m,l}$ we define
\begin{align*}
   A(i,j,X_{m,l})   &= \sum_{Y\subseteq X_{m,l}}(-1)^{|Y|}a_{iy_1}a_{y_1y_2}\cdots a_{y_kj};\\
   A'(i,j,X_{m,l})  &= \sum_{Y\subseteq X_{m,l}}(-1)^{|Y|}a_{iy_k}a_{y_ky_{k-1}}\cdots a_{y_1j}.
 \end{align*}
If $j\in X_{m,l}$ and $i\not\in X_{m,l}$ define
$$
B'(i,j,X_{m,l}) = \sum_{Y\subseteq X_{m,l}, y_1\ne j}c_Ya_{iy_k}a_{y_ky_{k-1}}\cdots a_{y_1j}
$$

\noindent where $c_Y = (-1)^{|Y|+1}$ if $Y\cap X_{m,j-m+1} = \emptyset$, and $c_Y = (-1)^{|Y|}$ if $Y\cap X_{m,j-m}\ne\emptyset$ (the $y_1\ne j$ condition makes this the complement of the first condition). To prove Lemma \ref{commutes} we use two lemmas. As explained in the proof of Lemma \ref{commutes}, it suffices to consider generators $a_{ij}$, $i<j$. Also, we write $\ast$ for $j=kp+1$. Recall the definition of the spanning arc $c_{ij}$ and the map $\chi:\mathscr S_{kp+1}\to\A_{kp}^L$ from Section \ref{SecSpanArcs}.

\begin{lem}\label{Sigma_n}
Given $1\le n\le k-1$ let $X_n^{(p)} = X_{(n-1)p+1,p}$. For $1\le i< j\le kp+1$ we have
$$
\phi_{\subpp{\s_n}}(a_{ij}) =
\begin{cases}
       a_{i+p,j+p} &\colon i,j\in X_n^{(p)}\\
       a_{i-p,j-p} & \colon i,j\in X_{n+1}^{(p)}\\
       B'(i+p,j-p,X_n^{(p)}) & \colon i\in X_n^{(p)}, j\in X_{n+1}^{(p)}\\
       a_{i-p,j} &\colon j > (n+1)p, i\in X_{n+1}^{(p)}\\
       a_{i,j-p} & \colon i\le(n-1)p, j\in X_{n+1}^{(p)}\\
       A(i,j+p,X_n^{(p)}) & \colon i\le(n-1)p, j\in X_n^{(p)}\\
       A'(i+p,j,X_n^{(p)}) & \colon j> (n+1)p, i\in X_n^{(p)}\\
       a_{ij} & \colon \textnormal{otherwise}
\end{cases}.
$$
\end{lem}

\begin{proof}
Define $\t_{m,l} = \s_m\s_{m+1}\cdots\s_{m+l-1}$ and let $\kappa_{m,l} = \t_{m+l-1,p}\t_{m+l-2,p}\cdots\t_{m,p}$. Note that $\kappa_{m,p}=\pp{\s_n}$ if $m=(n-1)p+1$. We may prove the result, therefore, by showing that for $i<j$ if $l\le p$ then

\begin{equation}\label{kappa}
\phi_{\kappa_{m,l}}(a_{ij}) =
  \begin{cases}
       a_{i+l,j+l} &\colon i,j\in X_{m,p}\\
       a_{i-p,j-p} & \colon i,j\in X_{m+p,l}\\
       B'(i+l,j-p,X_{m,l}) & \colon i\in X_{m,p}, j\in X_{m+p,l}\\
       a_{i-p,j} &\colon j\ge m+l+p, i\in X_{m+p,l}\\
       a_{i,j-p} & \colon i<m, j\in X_{m+p,l}\\
       A(i,j+l,X_{m,l}) & \colon i<m, j\in X_{m,p}\\
       A'(i+l,j,X_{m,l}) & \colon j\ge m+p+l, i\in X_{m,p}\\
       a_{ij} & \colon \textnormal{otherwise}
  \end{cases}.
\end{equation}

The proof of (\ref{kappa}) is by induction on $l$. For the case $l=1$, note that $\kappa_{m,1}=\t_{m,p}$. It is relatively straightforward to calculate, for $1\le m \le (k-1)p$ and $i<j$, that
\begin{equation}\label{tau}
\phi_{\t_{m,p}}(a_{ij}) =
  \begin{cases}
       a_{i+1,j+1} & \colon m\le i<j< m+p\\
       -a_{i+1,m} & \colon m\le i < j = m+p\\
       a_{mj} & \colon m+p=i<j\\
       a_{im} & \colon i<m<m+p=j\\
       a_{i,j+1}-a_{i,m}a_{m,j+1} & \colon i<m\le j< m+p\\
       a_{i+1,j}-a_{i+1,m}a_{m,j} & \colon m\le i< m+p<j\\
       a_{ij} & \colon \textnormal{otherwise}
  \end{cases}.
\end{equation}

  \begin{figure}[ht]
      \begin{tikzpicture}[scale=.75,>=stealth]
        \foreach \c in {3,10} \draw (\c,6) circle (60pt);
        \foreach \c in {3,10} \filldraw (\c-1.25,6) circle (1.5pt);
        \filldraw (4,6) circle (1.5pt);
        \filldraw (9.55,6) circle (1.5pt);
        \filldraw (3.15,6) circle (1.5pt)
                  (10.4,6) circle (1.5pt);
        \foreach \c in {3,10}  
          \draw   (\c+0.6,7.75) node {{\small $D$}};
        \draw[<-,thick] 
            (3.15,6)  ..controls (3.15,6.75) and (1.75,6.75) .. (1.75,6);
        \draw[->,thick] 
            (8.75,6)  ..controls (8.75,6.4) and (9.4,6.4) .. (9.4,6)
                     ..controls (9.4,5.8) and (9.7,5.8) .. (9.7,6)
                     ..controls (9.7,6.4) and (10.4,6.4) ..(10.4,6);
        \draw (10-.45,6) node[below] {{\tiny $m$}};
        \draw (3+1,6) node[below] {{\tiny $m+p$}};        
        \foreach \c in {3,10}
          \draw (\c-1.25,6) node[below] {{\tiny $i$}};
        \draw (3.15,6) node[below] {{\tiny $j$}}
              (10.4,6) node[below] {{\tiny $j+1$}};
        \draw[->,dashed]  (4,6) ..controls (4,6.15) and (3.75,6.27) ..(3.6,6.3);
        \draw[->,dashed]  (5.5,6) -- node[below] {\footnotesize $\tau_{m,p}$} (7.5,6);
        \foreach \c in {3,10} \draw (\c,1) circle (60pt);
        \foreach \c in {3,10} \filldraw (\c+1.45,1) circle (1.5pt);
        \filldraw (3.5,1) circle (1.5pt);
        \filldraw (8.65,1) circle (1.5pt);
        \filldraw (2.25,1) circle (1.5pt)
                  (9.5,1) circle (1.5pt);
        \foreach \c in {3,10}  
          \draw   (\c+0.6,2.75) node {{\small $D$}};
        \draw[->,thick] 
            (2.25,1)  ..controls (2.25,1.85) and (4.45,1.85) .. (4.45,1);
        \draw[->,thick] 
            (9.5,1)  ..controls (9.5,1.3) and (8.8,1.3) .. (8.8,1)
                     ..controls (8.8,.8) and (8.5,.8) .. (8.5,1)
                     ..controls (8.5,1.85) and (10.75,1.85) ..(11.45,1);
        \draw (10-1.35,1) node[below] {{\tiny $m$}};
        \draw (3+.5,1) node[below] {{\tiny $m+p$}};        
        \foreach \c in {3,10}
          \draw (\c+1.45,1) node[below] {{\tiny $j$}};
        \draw (2.25,1) node[below] {{\tiny $i$}}
              (9.5,1) node[below] {{\tiny $i+1$}};
        \draw[->,dashed]  (3.5,1) ..controls (3.5,1.15) and (3.25,1.27) ..(3.1,1.3);
        \draw[->,dashed]  (5.5,1) -- node[below] {\footnotesize $\tau_{m,p}$} (7.5,1);
      \end{tikzpicture}
      \caption{$\tau_{m,p}\cdot c_{ij}$, two possible cases}
      \label{FigTauCalc}
  \end{figure}
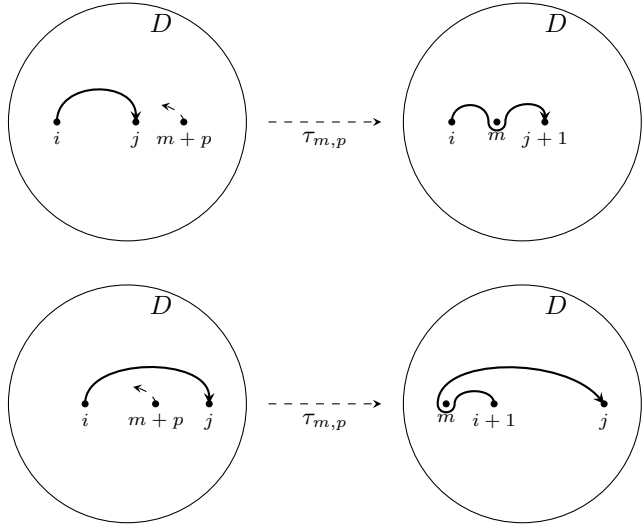

Indeed, the effect of $\tau_{m,p}$ is to move points $\{m,\ldots,m+p-1\}$ in $(D,P)$ one place to the right and the point at $m+p$ is carried through the upper half-disk to $m$. Figure \ref{FigTauCalc} shows $\t_{m,p}\cdot c_{ij}$ for two interesting cases in (\ref{tau}). Use the relation in Figure \ref{FigRelnPathAlg} at the point $m$, we get $\t_{m,p}\cdot c_{ij} = c_{i,j+1}-c_{im}c_{m,j+1}$ if $i<m\le j<m+p$, and $\t_{m,p}\cdot c_{ij} = c_{i+1,j}+c_{i+1,m}c_{mj}$ if $m\le i<m+p<j$. Applying the map $\chi$ gives the calculation in (\ref{tau}) for these cases. Verification of the other cases are left to the reader.

Since $X_{m,1}=\{m\}$, we have $A(i,j+1,X_{m,1})=a_{i,j+1}-a_{im}a_{m,j+1}$ and $A'(i+1,j,X_{m,1})=a_{i+1,j}-a_{i+1,m}a_{mj}$. Also, when $j=m+p$ the subsets considered for $B'(i+1,j-p,X_{m,1})$ must be empty, so it is $-a_{i+1,m}$. The other cases clearly agree with (\ref{kappa}) for $l=1$, proving the base case.

The argument for $l>1$ is handled in each case appearing in (\ref{kappa}).  We present the argument in the cases $i<m,j\in X_{m,p}$ and $j\ge m+p+l,i\in X_{m,p}$ and when $i\in X_{m,p}, j\in X_{m+p,l}$.

If $i<m, j\in X_{m,p}$ then

\begin{align*}
\phi_{\k_{m,l}}(a_{ij}) &= \phi_{\t_{m+l-1,p}}(\phi_{\k_{m,l-1}}(a_{ij}))\\
&= \sum_{{\scriptscriptstyle Y\subseteq \{m,\ldots,m+l-2\}}} (-1)^{|Y|} \phi_{\t_{m+l-1,p}}(a_{i,y_1}a_{y_1y_2}\cdots a_{y_k,j+l-1})\\
&= \sum_{{\scriptscriptstyle Y\subseteq \{m,\ldots,m+l-2\}}} (-1)^{|Y|} a_{iy_1}a_{y_1y_2}\cdots a_{y_{k-1}y_k}(a_{y_k,j+l}-a_{y_k,m+l-1}a_{m+l-1,j+l})\\
&= \sum_{{\scriptscriptstyle Y\subseteq \{m,\ldots,m+l-1\}}} (-1)^{|Y|} a_{i,y_1}a_{y_1y_2}\cdots a_{y_k,j+l}\\
&= A(i,j+l,X_{m,l}),
\end{align*}

The third equality uses (\ref{tau}) and holds because $l\le p$. 

The case $j\ge m+p+l,i\in X_{m,p}$ is very similar, but that the indices of generators appearing in the sum are descending, so we also use that $\phi_{\t_{m+l-1,p}}$ commutes with conjugation.

Finally, suppose $i\in X_{m,p},j\in X_{m+p,l}$. Note $j-(m+p)\le l-1$. If $j-m-p=l-1$, then by the preceding case

$$\phi_{\k_{m,l-1}}(a_{ij}) = A'(i+j-m-p,j,X_{m,j-m-p}).$$

We then have
{\small
\begin{align*}
&\phi_{\tau_{m + (j-m-p),p}}(A'(i+j-m-p,j,X_{m,j-m-p}))\\
&\quad =\sum_{{\scriptscriptstyle Y\subseteq \{m,\ldots,j - p - 1\}}} (-1)^{|Y|} \phi_{\t_{j-p,p}}(a_{i+j-m-p,y_k}a_{y_ky_{k-1}}\cdots a_{y_1,j})\\
&\quad = -a_{i+j-m-p+1,j-p}\\
&\quad + \sum_{{\scriptscriptstyle \overset{Y\subseteq \{m,\ldots,j - p - 1\}}{Y\ne\emptyset}}} (-1)^{|Y|} (a_{i+j-m-p+1,y_k} -a_{i+j+m-p+1,j-p}a_{j-p,y_k})a_{y_ky_{k-1}}\cdots a_{y_2y_1}a_{y_1,j-p}\\
&\quad = B'(i+l,j-p,X_{m,l}).
\end{align*}
}

If instead $j-m-p <l - 1$, and $l\le p$, we conclude the proof by checking
{\small
\begin{align*}
&\phi_{\tau_{m + l - 1,p}}(B'(i+l-1,j-p,X_{m,l-1})) \\
&\quad = \sum_{{\scriptscriptstyle \overset{Y\subseteq \{m,\ldots,m+l-2\}}{Y\cap X_{m,j-m-p}\ne\emptyset}}}(-1)^{|Y|}\phi_{\tau_{m + l - 1,p}}(a_{i+l-1,y_k}a_{y_ky_{k-1}}\cdots a_{y_1,j-p}) \\
&\qquad\quad - \sum_{{\scriptscriptstyle \overset{Y\subseteq \{m,\ldots,m+l-2\}}{Y\cap X_{m,j-m-p+1}=\emptyset}}}(-1)^{|Y|}\phi_{\tau_{m + l - 1,p}}(a_{i+l-1,y_k}a_{y_ky_{k-1}}\cdots a_{y_1,j-p}) \\
&\quad = \sum_{{\scriptscriptstyle \overset{Y\subseteq \{m,\ldots,m+l-2\}}{Y\cap X_{m,j-m-p}\ne\emptyset}}}(-1)^{|Y|}(a_{i+l,y_k} - a_{i+l,m+l-1}a_{m+l-1,y_k})a_{y_ky_{k-1}}\cdots a_{y_1,j-p} \\
&\qquad\quad - \sum_{{\scriptscriptstyle \overset{Y\subseteq \{m,\ldots,m+l-2\}}{Y\cap X_{m,j-m-p+1}=\emptyset}}}(-1)^{|Y|}(a_{i+l,y_k} - a_{i+l,m+l-1}a_{m+l-1,y_k})a_{y_ky_{k-1}}\cdots a_{y_1,j-p} \\
&\quad = B'(i+l,j-p,X_{m,l}). \qedhere
\end{align*}
}
\end{proof}

\begin{lem}\label{SimplifiedImages}
Fix $1\le i< j\le kp+1$ and define $\a_i = (n-1)p+r_i\in X_n^{(p)}$. We have the following equalities.
  \[\begin{array}{r l l} \psi(A(i,j+p,X_n^{(p)}))   &= \psi(a_{i,j+p} - a_{i\a_i}a_{\a_i,j+p})      & \colon i\le(n-1)p, j\in X_n^{(p)}\\
                  \psi^*(A'(i+p,j,X_n^{(p)})) &= \psi^*(a_{i+p,j} - a_{i+p,i}a_{i,j})       & \colon i\in X_n^{(p)}, j>(n+1)p\\
                \psi(B'(i+p,j-p,X_n^{(p)})) &= \psi(-a_{i+p,j-p} + \delta a_{i+p,i}a_{i,j-p}) & \colon i\in X_n^{(p)}, j\in X_{n+1}^{(p)},
  \end{array}\]
\noindent where $\delta\in\{-1,0,1\}$ is $0$ if $i=j-p$, and is the sign of $i-(j-p)$ otherwise.
\end{lem}

\begin{rem}It is possible to have $j=\ast$ only in the case that $j>(n+1)p$, hence the decoration $\psi^\ast$. This observation plays a role in Lemma \ref{commutes}.
\end{rem}
\begin{proof}[Proof of Lemma \ref{SimplifiedImages}] Each of the three cases involves a sum over subsets $Y\subseteq X_n^{(p)}$. 

In the case $i\le(n-1)p$, any $y_1<\a_i$ satisfies $r_{y_1}<r_i$ and $q_i<q_{y_1}$. Hence $\psi(a_{iy_1})=0$. Thus we restrict to subsets $Y\subseteq\{\a_i,\ldots,np\}$, i.e.
        \begin{align*}
        \psi(A(i,j+p,X_n^{(p)}))  &= \sum_{Y\subseteq \{\a_i,\ldots,np\}}(-1)^{|Y|}\psi(a_{iy_1}a_{y_1y_2}\cdots a_{y_k,j+p}).
        \end{align*}
For any $y_1\in\{\a_i+1,\ldots,np\}$ we get 
{\small
  \[\psi(a_{iy_1}-a_{i\a_i}a_{\a_iy_1}) = a_{q_i,q_{y_1}}\otimes a_{r_i,r_{y_1}}-(a_{q_i,n}\otimes 1)(1\otimes a_{r_i,r_{y_1}}) = 0,\]
}
\noindent and so 
{\small
        \begin{align*}
          &\psi(A(i,j+p,X_n^{(p)})                  & \\
          &\quad =\psi(a_{i,j+p}-a_{i\a_i}a_{\a_i,j+p}) &+\sum_{{\scriptscriptstyle \overset{Y\subseteq \{\a_i+1,\ldots,np\}}{Y\ne\emptyset}}}(-1)^{|Y|}\psi(a_{iy_1}-a_{i\a_i}a_{\a_iy_1})\psi(a_{y_1y_2}\cdots a_{y_k,j+p})\\
          &\quad =\psi(a_{i,j+p}-a_{i\a_i}a_{\a_i,j+p}).
        \end{align*}
}
In the remaining cases $i\in X_n^{(p)}$, and so $\a_i=i$. If $y_k>i$ then $r_{y_k}>r_{i+p}$ and $q_{i+p}>q_{y_k}$ so that $\psi(a_{i+p,y_k})=0$. Thus in these cases we restrict to $Y\subseteq\{(n-1)p+1,\ldots,i\}$. The argument for the second case proceeds analogously to the first.

In the third case, with $j\in X_{n+1}^{(p)}$, we must account for the condition $y_1\ne j-p$ in each summand. This causes the non-vanishing part of the sum to vary, depending on whether $i$ is larger than, equal to, or smaller than $j-p$. The $\delta$ in the statement of the lemma incorporates the three situations.  Noting that if $\emptyset \ne Y\subseteq\{(n-1)p+1,\ldots,i-1\}$ then $c_Y = -c_{Y\cup \{\a_i\}}$ (recall $\a_i=i$), the argument then proceeds analogously to the first.
\end{proof}

\begin{proof} [Proof of Lemma \ref{commutes}]
The statement holds when $\alpha$ is the identity braid. We prove for $1\le n < k$ that

$$\psi^\ast\circ\phi^\ast_{\subpp{\s_n}} = \left(\phi^\ast_{\s_n}\otimes \id\right)\circ\psi^\ast.$$

As the maps $B_{k} \to \text{Aut}(\A_k^L\otimes\A_p^L)$, given by $\alpha \mapsto \phi^\ast_\alpha\otimes\id$, and $B_k \to \text{Aut}(\A_{kp}^L)$, given by $\alpha \mapsto \phi^\ast_{\subpp\alpha}$, are homomorphisms, this suffices to prove the lemma.

Furthermore, for $\beta$ any braid, $\phi_\beta$ and $\psi$ both commute with conjugation, so we only need prove that 
      \begin{equation}\label{CommuteEqn}
      \psi^\ast(\phi^\ast_{\subpp{\s_n}})(a_{ij})=(\phi^\ast_\alpha\otimes\id)\psi^\ast(a_{ij})
      \end{equation}
for $i<j$, possibly $j=\ast$. We check (\ref{CommuteEqn}) for each case in the statment of Lemma \ref{Sigma_n}. 

In the first two cases both sides of (\ref{CommuteEqn}) equal $1\otimes a_{r_ir_j}$. 

When $j>(n+1)p, i\in X_{n+1}^{(p)}$, we could have $j=\ast$. Since $q_i=n+1$, we get $\psi^\ast(a_{i-p,\ast})=a_{q_i-1,\ast}\otimes a_{r_i,\ast}=(\phi^\ast_{\s_n}\otimes\id)\psi^\ast(a_{i,\ast})$. If $j\le kp$ then $\psi(a_{i-p,j}) = a_{q_i-1,q_j}\otimes x$ where $x=a_{r_ir_j}, 1$ or $0$ depending on the relation of $r_i$ to $r_j$. Again $q_i-1=n$, and $q_j>n+1$, so $a_{q_i-1,q_j}=\phi_{\s_n}(a_{q_iq_j})$, proving the statement. The case $i\le (n-1)p,j\in X_{n+1}^{(p)}$ is similar.

In the case that $\psi(a_{ij}) = A(i,j+p,X_n^{(p)})$ we have by Lemma \ref{SimplifiedImages} that
$$\psi(\phi_{\subpp{\s_n}}(a_{ij})) = \psi(a_{i,j+p}-a_{i\a_i}a_{\a_i,j+p}).$$

\noindent But since $q_i<q_{j+p} = n+1$ and $q_{\a_i}=n$, we see that
\begin{align*}
\psi(a_{i,j+p} - a_{i\a_i}a_{\a_i,j+p}) &= (a_{q_i,n+1} - a_{q_i,n}a_{n,n+1})\otimes x\\
&= (\phi_{\s_n} \otimes \id)(\psi(a_{ij})),
\end{align*}
\noindent where $x=a_{r_ir_j}$ if $r_i<r_j$, $x=1$ if $r_i=r_j$ and $x=0$ if $r_i>r_j$.

In the case that $\psi^\ast(a_{ij}) = A'(i+p,j,X_n^{(p)})$ (here $j=\ast$ is possible),
$$\psi^\ast(\phi_{\subpp{\s_n}}(a_{ij})) = \psi^\ast(a_{i+p,j}-a_{i+p,i}a_{ij}).$$

\noindent Then, as $q_{i+p} = n+1<q_j$ we get (replace $q_j$ with $\ast$ if $j=\ast$)
$$\psi^\ast(a_{i+p,j} - a_{i+p,i}a_{i,j}) = (a_{n+1,q_j} - a_{n+1,n}a_{n,q_j})\otimes x = (\phi^\ast_{\s_n} \otimes \id)(\psi^\ast(a_{ij})),$$
\noindent with $x$ either as before, or $x=a_{r_i,\ast}$ if $j=\ast$.

Finally, suppose $i\in X_n^{(p)},j\in X_{n+1}^{(p)}$. Then $q_{j-p}=q_i=n$ and $\a_i=i$. If $j-p<i$ then $r_j<r_i$. By Lemmas \ref{Sigma_n} and \ref{SimplifiedImages}
\begin{align*}
  \psi(\phi_{\subpp{\s_n}}(a_{ij})) = \psi(-a_{i+p,j-p}+\delta a_{i+p,i}a_{i,j-p})
            &= \psi(-a_{i+p,j-p}+a_{i+p,i}a_{i,j-p})\\
            &= (-a_{q_i+1,q_i}+a_{q_i+1,q_i})\otimes a_{r_i,r_j}\\
            &=0,
\end{align*}

\noindent which equals $(\phi_{\s_n}\otimes\id)\psi(a_{ij})$. If $j-p>i$ then 
  \[\psi(-a_{i+p,j-p}-a_{i+p,i}a_{i,j-p}) = -(a_{q_i+1,q_i}\otimes 1)(1\otimes a_{r_i,r_j}) = (\phi_{\s_n}\otimes\id)\psi(a_{ij}).\]
\noindent When $j-p=i$ then $r_i=r_j$ and $\psi(-a_{i+p,j-p}) = -a_{n+1,n}\otimes 1=(\phi_{\s_n}\otimes\id)\psi(a_{ij})$, and this finishes the proof.
\end{proof}


\section{Comments on augmentation rank and multiplicativity}
\label{SecComments}

The section is in two parts. First we prove Theorem \ref{ThmNNPlus1}, showing some cables of torus knots have augmentation rank less than bridge number. In the second part we discuss how this result, and some computational evidence, might fit into Conjecture \ref{ConjSuperMultipl}.

\subsection{Cables of $(n,n+1)$ torus knots}
\label{SecNNPlus1}

\newtheorem*{ThmNNPlus1}{Theorem \ref{ThmNNPlus1}}
\begin{ThmNNPlus1}Given $p>1$ and $n>1$, we have \[\ar(T((n,p),(n+1,1))) < np.\]
\end{ThmNNPlus1}

\begin{rem} The remarks at the end of Section \ref{SecAsSatelliteOp} imply that the bridge number of $T((n,p),(n+1,1))$ is $np$.
\end{rem}

\begin{figure}[ht]
      \begin{tikzpicture}[scale=.5,>=stealth]
        \foreach \c in {-6,5.2} \draw (\c,1) circle (60pt);
          \foreach \p in {-1.35,-.55} \draw (-6+\p,1) -- (-5.5+\p,1);
          \foreach \p in {-1.35,-.55} \draw (5.2+\p,1) -- (5.7+\p,1);
        \foreach \c in {-6,5.2}  
          \draw   (\c+2.1,1) node {{\large $\ast$}}
                  (\c+0.7,1) node {{\footnotesize $\ldots$}};
        \draw[->] (-6.3,1) ..controls (-6.3,1.7) and (-6.75,0.75) ..(-7.25,0.75)
                                  ..controls (-7.6,0.75) and (-7.6,0.9) ..(-7.6,1)
                                  ..controls (-7.6,2.5) and (-4.5,2) ..(-3.9,1);
        \draw[->] (4.9,1)  ..controls (4.9,1.7) and (4.1,0.3) ..(4.1,1);
        \draw     (-6.3,1) node[below] {{\tiny $s$}};
        \draw     (4.9,1) node[below] {{\tiny $s$}}
                  (4.1,1) node [above] {{\tiny $i$}};
        \draw (0,1) node {{\footnotesize $=\quad c_{p+s,\ast}\ +\quad$} $\displaystyle{\sum_{i=1}^p}$};
        \draw (7.4,1) node[right] {\footnotesize $c_{i,\ast}$};
      \end{tikzpicture}
      \caption{$\pp\tau\cdot c_{s\ast}$, $1\le s\le p$ as an element of $\mathscr S_{np+1}$}
      \label{FigTauInSn}
    \end{figure}

\begin{proof}
Let $\tau = \s_1\ldots\s_{n-1}\in B_n$ and set $\alpha = \tau^{n+1}$, which has the $(n,n+1)$ torus knot as its braid closure. We have $T((n,p),(n+1,1)) = K(\alpha,\gamma)$ for $\gamma=\s_1\ldots\s_{p-1}\in B_p$. Write $(\Phi_{\gamma\subsmallp(\alpha\subsmallp)}^L)_i$ for the $i^{th}$ row of $\Phi_{\gamma\subsmallp(\alpha\subsmallp)}^L$.

The structure of the proof is the following. By Theorem \ref{thm:RanknAugs} we must prove there is no homomorphism $\epsilon:\A_{np}\to\C$ such that $\epsilon(\Phi_{\gamma\subsmallp(\alpha\subsmallp)}^L)=\Delta(\gamma\smallp(\alpha\smallp))$. Note that, since $\overline\gamma$ is in the image of the inclusion $B_p\hookrightarrow B_{np}$ given by $\s_i\mapsto\s_i$, Theorem \ref{thm:ChainRule} implies that $(\Phi_{\gamma\subsmallp(\alpha\subsmallp)}^L)_i=(\Phi_{\subpp\alpha}^L)_i$ for $p<i\le np$. Hence, were such an $\epsilon$ to exist then $\epsilon((\Phi_{\subpp\alpha}^L)_i) = {\bf e}_{i}$ for $p < i \le np$. 

We will see that $(\Phi_{\overline\gamma}^L)_p={\bf e}_1$, implying that the entry $(\Phi_{\subpp\alpha}^L)_{1p}$ agrees with a diagonal entry of $\Phi_{\gamma\subsmallp(\alpha\subsmallp)}^L$. We then calculate that, for any $\epsilon$ satisfying $\epsilon((\Phi_{\subpp\alpha}^L)_i) = {\bf e}_{i}$ for $p < i \le np$, we must have $\epsilon((\Phi_{\subpp\alpha}^L)_{1p})=0$. This contradicts $\epsilon(\Phi_{\gamma\subsmallp(\alpha\subsmallp)}^L) = \Delta(\gamma\smallp(\alpha\smallp))$ and proves the result.

In consideration of the above, for the remainder of the proof $\epsilon:\A_{np}\to\C$ denotes a homomorphism with the property $\epsilon((\Phi_{\subpp\alpha}^L)_i) = {\bf e}_{i}$ for $p < i \le np$.

To prove that $\epsilon((\Phi_{\subpp\alpha}^L)_{1p})=0$ we first demonstrate, in {\bf I} below, that $\epsilon((\Phi_{\subpp\alpha}^L)_{1p})=-\epsilon(a_{p+1,p})$. This is followed in {\bf II} by a proof that $\epsilon(a_{p+1,p})=0$, which completes the proof of the theorem (the equality $(\Phi_{\overline\gamma}^L)_p={\bf e}_1$ is derived in the process).

{\bf I.} For $z\in\Z$, consider matrices $\Phi_{\subpp\tau^z}^L$ and partition them into an $n\times n$ array of $p\times p$ submatrices. In notation, define for $1\le i,j\le n$ the $p\times p$ matrix $\Psi_{ij}^z$ so that
    \[ \Phi_{\subpp\tau^z}^L = \begin{pmatrix}\Psi_{11}^z & \cdots & \Psi_{1n}^z \\ & \ddots & \\ \Psi_{n1}^z & \cdots & \Psi_{nn}^z\end{pmatrix}.\]

We claim that 
  \begin{enumerate}
    \item[(a)] the $(n-1)p\times(n-1)p$ submatrix $(\Psi_{ij}^1)_{i<n,j>1}$ is the identity matrix;
    \item[(b)] $\Psi_{n1}^1$ is the $p\times p$ identity matrix;
    \item[(c)] $\Psi_{nj}^1$ is the zero matrix for $j>1$;
    \item[(d)] finally, $(\Phi_{\overline\gamma}^L)_p = (1,0,\ldots,0)$.
  \end{enumerate}
Verification of the claim is left to the reader. As an example, (a) requires identities in $\mathscr S_{np+1}$ (which are passed through to $\A_{np+1}$ by $\chi$) similar to the identity in Figure \ref{FigTauInSn}, which can be used to calculate $\Psi_{1j}^1$ for $1\le j\le n$. Also, (d) can be deduced from (b) and (c) in the case that the $\Psi_{ij}^1$ are of size $1\times 1$.

By Theorem \ref{thm:ChainRule} we have $\Phi_{\subpp\tau^{z+1}}^L = \phi_{\subpp\tau}(\Phi_{\subpp\tau^{z}}^L)\Phi_{\subpp\tau}^L$. Thus by (a) and (c) above, for $1 \le j < n$,
  \begin{equation}
    \Psi_{i,j+1}^{z+1} = \phi_{\subpp\tau}(\Psi_{ij}^z).
    \label{eqn1}
  \end{equation}
Theorem \ref{thm:ChainRule} also shows $\Phi_{\subpp\tau^{z+1}}^L = \phi_{\subpp\tau^z}(\Phi_{\subpp\tau}^L)\Phi_{\subpp\tau^z}^L$. Hence by (b), (c)
  \begin{equation}
    \begin{aligned}
      \Psi_{nj}^{z+1}  &= \Psi_{1j}^z
    \end{aligned}
    \label{eqn2}
  \end{equation}
  
for all $1\le j\le n$, and, for $1\le i < n$, we have by (a) that
  \al{
    \Psi_{ij}^{z+1} &= \Psi_{i+1,j}^z + \phi_{\subpp\tau^z}(\Psi_{i1}^1)\Psi_{1j}^z.
  }

Taking $z=n$ above, equations (\ref{eqn1}) and (\ref{eqn2}) thus imply
  \al{
    \Psi_{ij}^{n+1} &= \phi_{\subpp\tau}^{-1}(\Psi_{i+1,j+1}^{n+1}) + \phi_{\subpp\tau^n}(\Psi_{i1}^1)\Psi_{nj}^{n+1}.
  }
  
  \begin{figure}[ht]
      \begin{tikzpicture}[scale=.6,>=stealth]
        \foreach \c in {-2} \draw (\c,1) circle (60pt);
        \foreach \c in {-2}
          \foreach \p in {-1.15,1.25} \filldraw (\c+\p,1) circle (1.5pt);
        \foreach \c in {-2}  
          \draw   (\c+2.1,1) node {{\large $\ast$}}
                  (\c+0.6,2.75) node {{\footnotesize $D$}};
        \foreach \c in {-2}
          \draw   (\c+.52,.9) node {{\small $\ldots$}};  
        \draw[->,thick] (-.75,1) ..controls (-.75,.3) and (-3.15,.3) ..(-3.15,1);
        \draw   (-.75,1) node[above] {{\tiny $s$}}
                (-3.15,1) node[above] {{\tiny $t$}};
        \draw   (-2,1-2.1) node[below] {\small $\chi^{-1}(w_{st})$};
      \end{tikzpicture}
      \caption{The spanning arc for $w_{st}$}
      \label{FigSentToZero2}
    \end{figure}
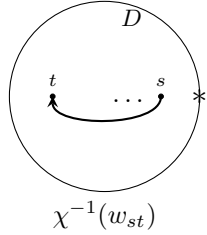

Hence $\epsilon(\Psi_{ij}^{n+1})=\epsilon(\phi_{\subpp\tau}^{-1}(\Psi_{i+1,j+1}^{n+1}))$ for $1\le j < n$, since $\epsilon(\Psi_{nj}^{n+1})={\bf 0}$ by assumption. Utilizing (\ref{eqn1}) and (\ref{eqn2}) again we find that, for $i\ge j$,
    \[\epsilon(\Psi_{ij}^{n+1}) = \epsilon(\phi_{\subpp\tau}^{-n+i}(\Psi_{n,j+(n-i)}^{n+1})) = \epsilon(\Psi_{nj}^{i+1}) = \epsilon(\Psi_{1j}^i).\]
Taking $s=1$ in Figure \ref{FigTauInSn}, we see that the $(1,p)$-entry of $\Psi_{11}^1$ is $\chi(c_{p+1,p})=-a_{p+1,p}$. And so $\epsilon((\Phi_{\subpp\alpha}^L)_{1p}) = \epsilon((\Psi^{n+1}_{11})_{1p}) = -\epsilon(a_{p+1,p})$, which we were to show in {\bf I}.

{\bf II.} We must show that $\epsilon(a_{p+1,p})=0$. To do so we consider $\phi^\ast_{\subpp\alpha}(a_{(n-1)p+1,\ast})$ in $\A_{np}^L\subset\A_{np+1}$ which, similar to above, we understand through the corresponding spanning arc (see Figure \ref{FigPhiLCalcPhi}). Our assumption that $\epsilon((\Phi_{\subpp\alpha}^L)_{(n-1)p+1})={\bf e}_{(n-1)p+1}$ means that if we write $\phi^\ast_{\subpp\alpha}(a_{(n-1)p+1,\ast})$ in the basis $\{a_{1,\ast},\ldots,a_{np,\ast}\}$ of $\A_{np}^L$ then $\epsilon$ sends all but $(n-1)p+1$ coefficient to zero. 

For $p< r \le np$, define $v_r\in\A_{np}$ such that $\chi^{-1}(v_r)$ is the spanning arc shown on the right in Figure \ref{FigPhiLCalcPhi}, which ends at $r$. Define $w_{st}$ so that (as shown in Figure \ref{FigSentToZero2}) $\chi^{-1}(w_{st})$ is contained in the lower half of $D_{np}$, and begins at $s$ and ends at $t$. If $s=t$ then we define $w_{st}=1$.

In {\bf I} we showed $\epsilon(\Phi_{ij}^{n+1}) = \epsilon(\Psi_{1j}^i)$ for any $i\ge j$. This has an important consequence for elements of the form $w_{ip+1,j}$. The entries of $\Psi_{1j}^i$ are computed from $\pp\tau^i\cdot c_{s,\ast}$ where $1\le s\le p$ (Figure \ref{FigTauInSn} shows the case $i=1$). Take $s=1$. Let $1\le i\le n-1$ and $1\le j=(q-1)p+r\le ip$ (for some $1\le r\le p$). Note this makes $i\ge q$. Then the $(1,r)$-entry of $\Psi_{1q}^i$ is $w_{ip+1,j}$. Our assumption on $\epsilon$ implies, only for $1<i\le n-1$, that
    \begin{equation}
    \epsilon(w_{ip+1,j})=\epsilon((\Psi_{iq}^{n+1})_{1r}) = \delta_{iq}\delta_{1r} = \delta_{(i-1)p+1,j},
    \label{EqnMod1Ws}
    \end{equation}
where $\delta$ is the Kronecker-delta. 

    \begin{figure}[ht]
      \begin{tikzpicture}[scale=.75,>=stealth]
        \foreach \c in {3,10} \draw (\c,1) circle (60pt);
          \foreach \p in {-1.35,-.75,-.15,1.25} \draw (3+\p,1) -- (3+\p+.2,1);
          \foreach \p in {-1.35,-.75,-.15} \draw (10+\p,1) -- (10.2+\p,1);
        \foreach \c in {3,10}  
          \draw   (\c+2.1,1) node {{\large $\ast$}}
                  (\c+0.6,2.75) node {$D$};
        \foreach \c in {3,10}
          \draw (\c+0.7,.9) node {{\small $\ldots$}};
          \filldraw (11.35,1) circle (1.5pt);
        \draw[->,thick] 
            (1.75,1)  ..controls (1.75,.85) and (2,.85) .. (2,1)
                      ..controls (2,1.8) and (4.6,1.8) .. (4.6,1)
                      ..controls (4.6,-.1) and (1.25,-.1) ..(1.25,1)
                      ..controls (1.25,2.3) and (4.8,2.25) ..(5.05,1.05);
        \draw   (1.75,.9) node[above]{{\tiny 1}};
        \draw[->,thick] 
            (8.75,1)  ..controls (8.75,.85) and (9,.85) .. (9,1)
                      ..controls (9,1.8) and (11.35,1.8) .. (11.35,1);
        \draw  (8.75,.9) node[above] {{\tiny 1}};
        \draw (3,-1.25) node[below] {{\small $\chi^{-1}(\phi^\ast_{\subpp\alpha}(a_{(n-1)p+1}))$}};
        \draw (10,-1.25) node[below] {{\small $\chi^{-1}(v_r)$}};
        \draw (11.35,1.1) node[below] {{\tiny $r$}};
      \end{tikzpicture}
      \caption{Row $(n-1)p+1$ of $\Phi_{\subpp\alpha}^L$}
      \label{FigPhiLCalcPhi}
    \end{figure}
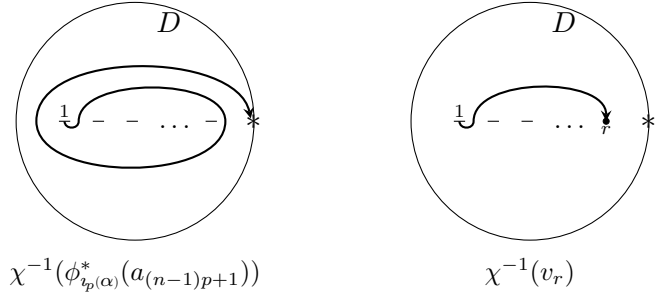

For $p<j\le np$, the coefficient of $a_{j,\ast}$ in $\phi^\ast_{\subpp\alpha}(a_{(n-1)p+1,\ast})$ can be written as
      \begin{equation}x_j := \langle \phi^{\ast}_{\subpp\alpha}(a_{(n-1)p+1,\ast}), a_{j,\ast} \rangle = \sum_{r=j}^{np}v_rw_{rj}.
      \label{eqnCurveExpansion}
      \end{equation}

{\bf Claim:} For $p < j \le np$, if $j=(n-i)p+1$ then $\epsilon(v_j)=(-1)^{i+1}$ and $\epsilon(v_j)=0$ otherwise.
\begin{proof}[Proof of Claim] The proof uses induction on $i$, proving the statement for each $(n-i)p+1\le j\le (n-i+1)p$. 

For $i=1$, by assumption $\epsilon(x_j)=0$ for $(n-1)p+1<j\le np$. Note that $x_{np} = v_{np}$. Thus $\epsilon(v_{np})=\epsilon(x_{np}) = 0$. This, applied to $j=np-1$, then $j=np-2$, and so on, implies that $\epsilon(v_j) = \epsilon(x_j) = 0$ for $(n-1)p+1<j\le np$. Furthermore, we get that $\epsilon(v_{(n-1)p+1}) = \epsilon(x_{(n-1)p+1}) = 1$.

Now suppose for some $1<i\le n-1$ that $(n-i)p+1\le j \le (n-i+1)p$. Assuming the claim holds for $v_{j'}$ with $j<j'$ we have
      \[0 = \epsilon(x_j) = \sum_{r=j}^{np}\epsilon(v_r)\epsilon(w_{rj}) = \epsilon(v_j) + \sum_{k=1}^{i-1}(-1)^{k+1}\epsilon(w_{(n-k)p+1,j}).\]
Recalling (\ref{EqnMod1Ws}), $\epsilon(w_{(n-k)p+1,j})=\delta_{(n-k-1)p+1,j}$ (provided $n-k>1$), and thus $\epsilon(v_j) = 0$ provided $j\ne (n-i)p+1$. When $j=(n-i)p+1$ we get that 
      \[0 = \epsilon(v_j) + (-1)^{i}\epsilon(w_{(n-i+1)p+1,(n-i)p+1}) = \epsilon(v_j) + (-1)^{i}\]
\noindent as claimed.
\end{proof}

We finish the proof by considering $\langle \phi^{\ast}_{\subpp\alpha}(a_{(n-1)p+1,\ast}), a_{p,\ast} \rangle$; the spanning arc corresponding to $\phi^{\ast}_{\subpp\alpha}(a_{(n-1)p+1,\ast})$ indicates a small difference to the previous coefficients. We have
      \[x_p := \langle \phi^{\ast}_{\subpp\alpha}(a_{(n-1)p+1,\ast}), a_{p,\ast} \rangle = \sum_{r=p+1}^{np}v_rw_{rp}.\]
Applying our previous claim, (\ref{EqnMod1Ws}), and $w_{p+1,p}=-a_{p+1,p}$ we have
        \begin{align*}
        0 = \epsilon(x_p)   = \sum_{i=1}^{n-1}(-1)^{n-i+1}\epsilon(w_{ip+1,p})  &= (-1)^n\epsilon(w_{p+1,p})\\
                                                                                &= (-1)^{n-1}\epsilon(a_{p+1,p}).
        \end{align*}
This implies the desired result and finishes the proof of the theorem. 
\end{proof}

\subsection{Augmentation rank does not multiply}
\label{SecMultiplic}

As discussed in Section \ref{SecAsSatelliteOp} the braid satellite $K(\alpha,\gamma)$ depends only on $\gamma$ and the closure $\hat{\alpha}$, if $\alpha$ has minimal braid index. Letting $\omega$ denote the writhe of such a braid, we write $P$ for the closure $\widehat{\Delta^\omega\gamma}$, as in Section \ref{SecAsSatelliteOp}.

\begin{prop}For any braid $\alpha$ with $K=\hat{\alpha}$ and $\gamma\in B_p$, there are $p$ KCH irreps $\s:\pi_{K(\alpha,\gamma)}\to\text{GL}_r\C$ for each KCH irrep $\rho:\pi_K\to\text{GL}_r\C$. In particular, $\ar(K(\alpha,\gamma)) \ge \ar(K)$.
\label{ThmCompanionRank}
\end{prop}
\begin{proof}Consider a neighborhood $n(K)$ of $K$ that contains $K(\alpha,\gamma)$. Write $T=\partial(n(K))$. Choose the basepoint of $\pi_{K(\alpha,\gamma)}$ on $T$. Then inclusion makes $\pi_1(T)$ a subgroup and $\pi_{K(\alpha,\gamma)}$ is isomorphic to the product of $\pi_K$ and $\pi_P$ amalgamated along $\pi_1(T)$.

Let $m_1$ be the meridian of $K$ determined by a based loop contained in $T$ that is contractible in $n(K)$. Suppose that $\rho:\pi_K\to\text{GL}_r\mathbb C$ is an irreducible KCH representation with $\widetilde M = \rho(m_1) = \text{diag}[\widetilde\mu_0,1,\ldots,1]$ for some $\widetilde\mu_0\in\mathbb C\setminus\{0\}$. Choose any $p^{th}$ root $\mu_0$ of $\widetilde\mu_0$. 

Consider a collection of meridians $m_1,\ldots,m_r$ of $K$ that generate $\pi_K$. For each $1\le i\le r$ there are $p$ meridians $m_{i1},\ldots,m_{ip}$ of $K(\alpha,\gamma)$ such that $m_{i1}m_{i2}\ldots m_{ip} = m_i$. Set $\s(m_{1j}) = \text{diag}[\mu_0,1,\ldots,1] = M$ for $1\le j\le p$. Then, for each $1< i\le r$ find $w_i\in\pi_K$ so that $m_i = w_im_1w_i^{-1}$ and set $\s(m_{ij}) = \rho(w_i)M\rho(w_i)^{-1}$ for $1\le j\le p$.

Due to the braid pattern of $K(\alpha,\gamma)$, $\pi_{K(\alpha,\gamma)}$ has a presentation so that each relation has the form $xm_{i,j}x^{-1} = w_im_{1,k}w_i^{-1}$, where $x$ is a word in $\{m_{i1}^{\pm},\ldots,m_{ip}^{\pm}\}$ and $1\le j, k\le p$, $1\le i\le r$. Thus $\s:\pi_{K(\alpha,\gamma)}\to\text{GL}_r\mathbb C$ is a well-defined KCH representation. Moreover, the image of $\s$ contains that of $\rho$, implying it is irreducible and that $\ar(K(\alpha,\gamma))\ge \ar(K)$.
\end{proof}

We remark that $\ar(K(\alpha,\gamma))\ge \ar(P)$ also, for $P=\widehat{\Delta^{2\omega}\gamma}$. This follows from Proposition \ref{PropAsSatelliteOp} and the existence of a surjection $\pi_{K(\alpha,\gamma)}\to\pi_P$, preserving peripheral structures (see Proposition 3.4 in \cite{SW}, for example). 

Oddly, the product $\ar(K)\ar(P)$ does not relate well to $\ar(K(\alpha,\gamma))$: from Theorem \ref{ThmNNPlus1} we find examples where $\ar(K(\alpha,\gamma))<\ar(K)\ar(P)$ and from Theorem \ref{main} there are examples with $\ar(K(\alpha,\gamma))>\ar(K)\ar(P)$ (take $\alpha=\s_1^3$ and $\gamma=\s_1^{-5}$, for example). However, to our knowledge the statement of Conjecture \ref{ConjSuperMultipl} could hold. 

There are cases where $\ar(K(\alpha,\gamma))$ is strictly larger than $\ar(K)\ar(\hat{\gamma})$. One example can be found from the $(2,11)$-cable of the $(2,5)$ torus knot. From computer-aided computations, we have a solution to (\ref{eqn:FindingAugs}) for $\alpha = \s_1^5\in B_2$ and $\gamma=\s_1\in B_2$, showing that $\ar(K(\s_1^5,\s_1)) = 4$, even though $\ar(\hat{\s_1^5})=2$ and $\ar(\hat{\s_1})=1$. Unfortunately, other examples of cables of torus knots (not covered by Theorems \ref{main} and \ref{ThmNNPlus1}) seem outside our computational abilities.

We end with computational observations and a question. By the inequalities in (\ref{cor:DimBound}) if a knot has bridge number less than its minimal braid index $n$, it cannot have augmentation rank equal to $n$. Take a minimal index braid representative of such a knot, and multiply that braid by successively higher powers of $\Delta^2\in B_n$, testing in each instance if the closure has augmentation rank equal to $n$. In examples, the power of $\Delta^2$ need not be very high, compared to $n$, before a braid with augmentation rank $n$ appears. Also, once such an augmentation appears, it seems to persist. 

Dehornoy introduced a total, left-invariant order on $B_n$. By Theorem \ref{ThmNNPlus1} the closure of $\s_1\smallp(\s_1^3\smallp)$ has augmentation rank less than 4. In comparison, $\s_1\smallp(\s_1^5\smallp)$ is larger in Dehornoy's order on $B_4$ and, as mentioned, has augmentation rank 4. 

The relation in the order of a braid to powers of a full twist has been shown to carry significance for the braid closure. In fact, it was shown in \cite{MN} that there is a constant $m_n$ such that if $\alpha>\Delta^{2m_n}$ (or $\alpha^{-1}>\Delta^{2m_n}$) then $\alpha$ does not admit one of the Birman-Menasco templates, and thus is a minimal index representative of $K=\hat\alpha$ by the MTWS \cite{BM_MTWS}. Perhaps there is a similar result for augmentation rank.

\newtheorem*{ques}{Question}
\begin{ques}For a given braid index $n$, is there a number $m_n$ so that $\ar(\hat\alpha)=n$ for any $\alpha\in B_n$ (with connected closure) greater than $\Delta^{2m_n}$ in Dehornoy's order?
\end{ques}

\bibliography{AugsCables_refs}
\bibliographystyle{alpha}
\end{document}